\documentclass[11pt,reqno]{amsart}

\usepackage{natbib}
\setcitestyle{square, comma, numbers,sort&compress, super}

\usepackage[a4paper,vmargin={3.5cm,3.5cm},hmargin={2.5cm,2.5cm},bottom=30mm]{geometry}

 \newtheorem{step}{Step}

\newtheorem{algorithm}{Algorithm}

\newcommand{\RR}{\mathbb{R}}
\newcommand{\NN}{\mathbb{N}}

\newcommand{\ZZ}{\mathbb{Z}}

\newcommand{\Bb}{\mathcal{B}}
\newcommand{\ee}{\mathrm{e}}

\newcommand{\De}{\mathrm{d}}

\usepackage{booktabs} 

\usepackage{array} 
\usepackage{paralist} 
\usepackage{verbatim} 
\usepackage{subfig} 
\usepackage[T1]{fontenc} 

\def\1{{\mathchoice {\rm 1\mskip-4mu l} {\rm 1\mskip-4mu l}
{\rm 1\mskip-4.5mu l} {\rm 1\mskip-5mu l}}}

\usepackage{pxfonts}
\usepackage{dsfont}

\usepackage{hyperref}
\hypersetup{ colorlinks=true, linkcolor=blue, citecolor=blue,
  filecolor=blue, urlcolor=blue}
\usepackage{tikz}

\DeclareMathOperator{\T}{\textbf{T}}
\DeclareMathOperator{\R}{\mathbb{R}}

\DeclareMathOperator{\Z}{\mathbb{Z}}
\DeclareMathOperator{\N}{\mathbb{N}}
\usepackage{mathtools}

\newcommand{\var}{\mathsf{Var}}

\renewcommand{\T}{\mathbb{T}}

\DeclareMathOperator{\e}{e}
\newcommand{\la}{\left\langle}
\newcommand{\ra}{\right\rangle}

\newcommand{\eq}[1]{\begin{equation#1}}
\newcommand{\eeq}[1]{\end{equation#1}}

\newtheorem{theorem}{Theorem}
\newtheorem{proposition}[theorem]{Proposition}

\newtheorem{lemma}[theorem]{Lemma}

\theoremstyle{definition}
\newtheorem{definition}{Definition}[section]

\usepackage{amssymb}
\numberwithin{equation}{section}

\title[Odometers of Divisible Sandpile Models: Scaling Limits, iDLA and Obstacle Problems. A Survey]{Odometers of Divisible Sandpile Models: Scaling Limits, iDLA and Obstacle Problems. A Survey}
\author[W.~M.~Ruszel]{Wioletta M.~Ruszel}
\address{TU Delft (DIAM), Building 28, van Mourik Broekmanweg 6, 2628 XE, Delft, The Netherlands}
\email{W.M.Ruszel@tudelft.nl}
\thanks{Acknowledgements: The author would like to thank Leandro Chiarini, Alessandra Cipriani and Rajat Hazra for useful comments on the previous versions of the draft and Jan de Graaff for the pictures.}

\begin{document}
\begin{abstract}
The divisible sandpile model is a fixed-energy continuous counterpart of the Abelian sandpile model. We start with a random initial configuration and redistribute mass deterministically. Under certain conditions the sandpile will stabilize. The associated odometer function describes the amount of mass emitted from each vertex during stabilization. In this survey we describe recent scaling limit results of the odometer function depending on different initial configurations and redistribution rules. Moreover we review connections to the obstacle problem from potential theory, including the connection between odometers and limiting shapes of growth models such as iDLA. Finally we state some open problems.
\end{abstract}
\maketitle

\tableofcontents

\section{Introduction}

The divisible sandpile model is a continuous height  and fixed energy counterpart of the Abelian sandpile model (ASM)  \citet{BTW}.  It was introduced by \citet{LevPer} to study scaling limits of two growth models: the rotor-router model and internal diffusion limited aggregation.
The model is interesting by itself and has further many nice connections to free boundary PDE, potential theory, algebraic geometry or computational complexity, see \citet{LapGrowth} and references therein. 

Consider a locally finite and undirected graph  $G=(V,E)$. A divisible sandpile configuration $s:V\rightarrow \R$ assigns to each $x\in V$ a \textit{height} $s(x)$ representing some \textit{mass} in case $s(x)>0$ and a \textit{hole} for $s(x)<0$.  At each time step, if $s(x)>1$, then redistribute the excess $(s(x)-1)^+$ among the neighbours. Repeat  for all such unstable vertices $x$. We will make two different choices on how to do that and capture the choice by $\Delta_V^{\star}$.
In the case where the excess is redistributed equally to nearest neighbours, we have $\Delta_V^{\star}$, where $\Delta_V^{\star}=\Delta$ the usual graph Laplacian. For the other case, consider a discrete fractional Laplacian $\Delta_V^{\star}=\Delta^{\alpha/2}$ capturing the redistribution to far away neighbours proportionally to the transition probabilities of a long-rang random walk. Note that once $s$ is fixed the evolution is deterministic. This model is Abelian in the sense that the final configuration does not depend on the toppling order (we can choose parallel toppling for simplicity).
Call
\[
u_n(x) = \text{total amount of mass emitted before time $n$ from $x$ to the neighbours}
\]
It is easy to see that $u_n$ is increasing in $n$ and hence will converge to some
\[
u:V\rightarrow [0,\infty].
\]
Define the following dichotomy: either for all $x\in V$, $u(x)<\infty$ so the divisible sandpile configuration will stabilize or $u(x)=\infty$ for all $x$ which means that $s$ explodes.  The first question one might ask : 
\begin{center}
Under which conditions on $s$ and $G$ does the sandpile stabilize?
\end{center}

\noindent
\textbf{Dichotomy between stabilization and explosion:}

When $|V|<\infty$ then it was shown in Lemma 7.1 of \citet{LMPU} that for any initial configuration $s$ satisfying $\sum_{x\in V}s(x)\leq |V|$ the sandpile will stabilize and explode otherwise. In the infinite volume case assume that $(s(x))_{x\in V}$ is a collection of i.i.d. random variables. Then it turns out that the mean or density in the physical sense (mass per unit volume) $\rho=\mathbb{E}(s(o))$ is the main parameter to determine stabilization. In Lemma 4.2. of \citet{LMPU} it was proven that indeed for $\rho<1$, the divisible sandpile stabilizes and in Lemma 4.1 for $\rho >1$ explodes almost surely and on the boundary case $\rho=1$ (Theorem 1.1.) the authors in \citet{LMPU} proved under a finite variance condition that $s$ does not stabilize a.s. This result was further extended to initial configurations with possibly no mean (Lemma 1) or variance (Theorem 2+3) in \citet{CHRheavy}.

Note the interesting fact that in the Abelian sandpile model, the density alone is not enough to determine whether the sandpile stabilizes or explodes. 
In fact, for $V=\Z^d$ there exist infinite volume measures $\mu$ with mean $\rho \in (\rho_c, 2d)$ such that $\mu$ is not stabilizable (has not probability one on configurations which will stabilize), compare Theorem 5.1 in \citet{AnneFrank}.
 
Recall that the ASM is a toy for model for self-organized criticality (SOC). A SOC system drives itself into a critical state (typically having power-law decay of correlations) without fine-tuning any parameters (temperature pressure, etc.).  
Let us point out the debate about  whether self-organized criticality intrinsically involves tuning a parameter, namely the density. It was claimed in a series of papers, see Section III of \citet{Vesp} and subsequent papers. In fact, the authors in \citet{VespDenCon} in Section II introduced the concept of a fixed-energy Abelian sandpile (FES) where mass is not lost at the boundary, contrary to the ASM.  They conjecture that the critical densities of a FES and ASM are the same. This would imply that the power-laws in ASM are coming from a real phase transition and not reminiscent of a SOC state. Although the critical values they found are close, it could be shown in large scale simulations that they are not the same, compare Theorem 1 of \citet{AnneDrive}. 

Given that we know the sandpile will stabilize on a finite graph, e.g. under the condition $\sum_{x\in V}s(x)\leq |V|$, another question is:
\begin{center}
What is the law of $u$ and it's scaling limit?
\end{center}

\noindent
\textbf{Law of the odometer and it's scaling limit:}

Consider the initial configuration $s$ given by
\[
s(x) = 1+ \sigma(x) - \frac{1}{n^d}\sum_{z\in V}\sigma(z).
\]
Then we have trivially that $\sum_{x\in V} s(x)= |V|$.
The collection  of random variables  $(\sigma(x))_{x\in V}$ will be either i.i.d. and satisfy a finite second moment assumption,  correlated Gaussian random variables or i.i.d. and in the domain of attraction of $\alpha$-stable random variables. Using the redistribution $\Delta_V^{\star}$ defined above note that $u$ satisfies
\[
\Delta_V^{\star} u = 1- s.
\]
We see that roughly
\[
\begin{split}
u &= (-\Delta_V^{\star})^{-1}\left (\sum_{x\in V} \sigma(x) - \frac{1}{|V|} \sum_{z\in V}\sigma(z) \right) \approx (\Delta_V^{\star})^{-1}\left (\sum_{x\in V} \sigma(x) \right ).
\end{split}
\]
Intuitively we can guess the scaling limit using the following considerations. Let $\sigma(x) \sim N(0,1)$ for all $x\in V$, independently, and $\alpha \in (0,2]$
then the authors in \citet{CHR17} in Theorem 1+2 and Theorem 3.4 of \citet{LongRange} prove that for $V=\Z^d_n$ the discrete torus of length $n$,
\[
(\Delta_V^{\star})^{-1}\left (\sum_{x\in V} \sigma(x) \right ) \longrightarrow \Delta^{- \alpha/2} W
\]
in law to the continuum fractional field on $\T^d$. $W$ denotes spatial white noise and the notation $\Delta^{- \alpha/2} W$ is adapted from \citet{LSSW} equation 1.1.
Special cases include: the Gaussian free field ($\alpha=1$) and the membrane model or bi-Laplacian field ($\alpha=2$). 
This convergence was first conjectured for a sub-case ($\sigma$'s having finite variance and $\Delta_V^{\star}=\Delta$ the usual graph Laplacian) in \citet{LMPU} after Proposition 1.3. It is notable that we cannot go beyond membrane models by playing with the long-range parameter $\alpha$. 

In fact, to obtain fractional Gaussian fields with regularity $\Delta^{-(1+\delta)} W$, for $\delta>0$ consider correlated initial Gaussian variables $(\sigma(x))_{x\in V}$ such that roughly
\[
(\Delta_V^{\star})^{-1}\left (\sum_{x\in V} \Delta^{-\delta} \sigma(x) \right ) \longrightarrow \Delta^{- (1+\delta)} W.
\]
This is a special case of the results in, Theorem 1 in \citet{JanFourier}, where the authors constructed a family of fractional fields specified by Fourier multipliers $\widehat{K}$ via defining an appropriate inverse covariance kernel for the Gaussian correlated random variables $(\sigma(x))_{x\in V}$. 

Can we go beyond Gaussianity? Yes, by imposing that $\sigma$'s will not satisfy a second moment condition but rather being in the domain of attraction of a $\alpha$-stable distribution. Then the authors proved in Theorem 5 of \citet{CHRheavy} that the limiting odometer field is in fact a $\alpha$-stable random field on $\T^d$. (Note that the parameter $\alpha$ from $\alpha$-stable distributions is not the same $\alpha$ which plays a role in the definition of a discrete fractional Laplacian and long-range random walk. It  will be clear from the context that they are not the same parameters.)\\

\noindent
\textbf{Obstacle problem and iDLA:}

An equivalent description of the odometer function can be achieved via solutions of an obstacle problem. Given an obstacle $\gamma$ satisfying $\Delta_V^{\star} \gamma = 1-s$ find a function $v$ such that $v(x)=\inf \{f(x)|f\geq \gamma, (-\Delta_V)^{\star}f \leq 0 \}$. Then the odometer is equal to
\[
u_{\infty} = v - \gamma.
\]
Note that the scaling limits obtained in Section \ref{sec:scale} are, by abuse of notation, scaling limits of the \textit{obstacle} and not the full odometer function. Define the formal field on the discretized torus:
\[
\Xi^u_n = \sum_{z\in \T^d_n} u^{\star}(z\cdot n) \1_{B(z,1/2n)}(x)
\]
for $x\in \T^d$.
However, the formal field given $\Xi^u_n$ whose convergence we study involves to take an inverse Laplacian resp. fractional Laplacian with zero eigenvalue. We need to remove it by imposing zero mean test functions $f$. Hence in law:
\[
\langle \Xi^u_n , f \rangle = \langle \Xi^{\gamma}_n , f \rangle
\]
where $\Xi_n^{\gamma}$ is the discretized field involving $\gamma$ and not $v-\gamma$. $v$ turns out to be related to the maximum of $\gamma$. One open question is what is the scaling limit of this maximum. Knowing the full scaling limit of the odometer allows to answer questions related to limiting shapes of growths models such as iDLA, rotor-router models and divisible sandpiles starting from a pile of $n$ chips at the origin. As we will explain in Section \ref{sec:iDLA}, the limiting shape will be related to determining the so-called non-coincidence set $\{x\in \R^d| v(x)>\gamma(x)\}$ for the obstacle problem.\\

The paper is organized as follows. In Section \ref{sec:not} we set the notation and preliminaries. We will introduce abstract Wiener spaces, sandpile models, stabilization and the reformulation of the odometer function in terms of solutions of an obstacle problem. The subsequent Section \ref{sec:law} will give an overview over results determining the law of the odometer on the discrete torus subjected to different redistribution rules and initial configurations. Section \ref{sec:scale} deals with scaling limit results for the Gaussian case and gives the main ideas of the proofs. In the following Section \ref{sec:stable} we state scaling limit results for the non-Gaussian case, when dealing with initial configurations being in the domain of attraction of $\alpha$-stable random variables. Section \ref{sec:iDLA} makes the connection between limiting shapes for growth models and non-coincidence sets for the odometer function. The last Section \ref{sec:dis} discusses open problems to explore.

\section{Notation and Preliminaries}\label{sec:not}

\subsection{Fourier analysis and abstract Wiener spaces}
\subsubsection{Fourier analysis on the torus and continuum fractional Laplacians}
\label{subsec-fourier-torus}
We will from now on assume that $V=\Z^d_n$.  Let $\T^d$ be the $d$-dimensional torus, alternatively  $\frac{\R^d}{\Z^d}$ or as $[-\frac12,\,\frac12)^d\subset\R^d$. The discrete torus with side length $n$ is denoted by $\Z_n^d:=[-\frac{n}{2},\,\frac{n}{2}]^d\cap \Z^d$. Finally $\T_n^d:=[-\frac12,\,\frac12]^d\cap (n^{-1}\Z)^d$ is the discretization of $\T^d$.  Moreover call $B(z,\,r)$ a ball centered at $z$ of radius $r>0$ in the $\ell^\infty$-metric. 
Define the inner product for  $\ell^2(\mathbb{Z}^d_n)$:

\[
	\langle {f,g} \rangle 
	=
	\frac{1}{n^d} \sum_{z \in \mathbb{Z}^d_n} 
	f(z) \overline{g(z)}.
\]
We will discuss two cases involving the classical graph Laplacian $\Delta_g$ and discrete fractional Laplacian $-(-\Delta_n)^{\alpha/2}$ for $\alpha \in (0,\infty)$.  
$\Delta_g$ is defined by
\[
\Delta_g f(x) = \frac{1}{2d}\sum_{\|y - x\|=1}\left (f(y) - f(x) \right )
\]
and
\begin{equation}\label{fracLap}
-(-\Delta_n)^{\alpha/2} f(x) = \left (\sum_{y\in \Z^d_n} f(y) p^{\alpha}_n(x-y) \right) - f(x),
\end{equation}
where $p_n^{\alpha}:\Z^d_n \times \Z^d_n \rightarrow [0,1]$ is the transition kernel for a long-range random walk defined by
\begin{equation}\label{def-random-walk-in-zdn}
	p^{(\alpha)}_n(0,x): =
	c^{(\alpha)} 
	\sum_{\substack{z \in  \mathbb{Z}^d\backslash \{0\} \\ z \equiv x 
	\!\!\!\!\! 
	\mod \mathbb{Z}^d_n }} \frac{1}{\|z\|^{d+\alpha}} ,
\end{equation}
where $c^{(\alpha)}=(\sum_{z \in  \mathbb{Z}^d \backslash \{0\}} \frac{1}{\|z\|^{d+\alpha}})^{-1}$ is the constant such that $\sum_{x \in \mathbb{Z}^d_n} p^{(\alpha)}_n(0,x) = 1$ and $x \equiv z \mod \mathbb{Z}^d_n$ denotes that $x_j \equiv z_j \mod n$ for all $j \in \{1,\dots,d\}$. With abuse of notation we will from now on write $p^{(\alpha)}_n(x):=p^{(\alpha)}_n(0,x)$.

Consider the Fourier basis given by the eigenfunctions of both Laplacians $\{\psi_w\}_{w\in \mathbb{Z}^d_n}$ with
\begin{equation}\label{def-fourier-basis-discrete}
	\psi_w(z)=\psi^{(n)}_w(z):=\exp \bigg(- 2\pi i z\cdot\frac{w}{n}\bigg).
\end{equation}

Given $f \in \ell^2(\mathbb{Z}^d_n)$, we define its discrete Fourier transform by
\[
	\widehat{f}(w) 
=
	\langle  f, \psi_{w}\rangle
= 
	\frac{1}{n^d}\sum_{z \in \mathbb{Z}^d_n} 
	f(z) \exp \bigg(-2\pi i z\cdot\frac{w}{n}\bigg)  
\]
for $w\in \mathbb{Z}^d_n$.

The eigenvalues   $\{\lambda_w\}_{w\in \mathbb{Z}^d_n}$ corresponding to the classical graph Laplacian are given by
\begin{equation}\label{eq:eigenvalues}
 \lambda_w:=-4\sum_{i=1}^d\sin^2\biggl(\frac{\pi w_i}{n}\biggr ),
\end{equation}
whereas for the discrete fractional Laplacian, see equation 4.2 in \citet{LongRange},
\[
    \lambda^{(\alpha,n)}_w = 
    -c^{(\alpha)} \frac{1}{n^{d+\alpha}} 
    \sum_{x \in \frac{1}{n} \mathbb{Z}^d \setminus \{0\}} 
    \frac{\sin^2(\pi x \cdot w )}{\|x\|^{d+\alpha}}, 
\]
where $c^{(\alpha)}$ is just the normalising constant of the associated long range-random walk in $\mathbb{Z}^d$. In equation 4.6 of \citet{JanFourier} we consider the following definition of discrete fractional Laplacian 
\[
(-\Delta_g)^{-a} f(\cdot) = \sum_{v\in \Z^d \setminus \{o\}} (-\lambda_v)^{-s} \widehat{f}(v)\psi_v(\cdot)
\]
for $f\in \ell^2(\Z^d_n)$.
Note that in contrast to \eqref{fracLap} there is no random walk representation.

Similarly, if $f, g \in L^2(\mathbb{T}^d)$  denote by
\[
	(f,g)_{L^2(\mathbb{T}^d)}:= \int_{\mathbb{T}^d} f(z) \overline{g(z)}
	\text{d}z
\]
the inner product.
Consider the Fourier basis $\{\phi_{\xi}\}_{\xi \in \mathbb{Z}^d}$ of $L^2(\T^d)$ given by 
$$\phi_{\xi}(x):=\exp(-2\pi i \xi \cdot x)$$ 
so
\[
	\widehat{f}(\xi): =(f,\phi_{\xi})_{L^2(\mathbb{T}^d)} = 
	\int_{\mathbb{T}^d} f(z) \ee^{- 2 \pi i \xi \cdot z} \text{d}z.
\]
It is important to notice that for $f \in C^\infty(\mathbb{T}^d)$, if we define $f_n: \mathbb{Z}^d_n \to \mathbb{R}$ by $f_n(z) := f( {z}/{n})$, then for all $\xi \in \mathbb{Z}^d$, $\widehat{f_n}(\xi) \to \widehat{f}(\xi)$
as $n\rightarrow \infty$. 

Finally, we write $C^\infty(\T^d)/\sim$ for the space of smooth functions modulo the equivalence relation of differing by a constant. 
\begin{definition}
Let $a\in \R$ we define the fractional operator $(-\Delta)^a$ as
\[
(-\Delta)^a f(\cdot) = \sum_{v\in \Z^d \setminus \{o\}} \|v\|^{2a} \widehat{f}(v) \phi_v(\cdot)
\]
for $f\in L^2(\T^d)$.
\end{definition}

Remark that
\[
(-\Delta)^a \phi_v(\cdot) = \| v\|^{4a} \phi_v(\cdot).
\]

Note that most results about continuum fractional Laplacians are obtained for $a \in (0,2)$, compare also different representations in Theorem 1 of  \citet{Kwasnicki15}.
The reason is that for $\alpha \in (0,2)$ and $f \in C^{\infty}(\mathbb{R}^d)$, $f(\cdot + e_i)=f(\cdot)$ for all $i =1, \dots,d$ and $\{e_i\}_{i=1}^d$  the canonical basis of $\mathbb{R}^d $ we can write
\begin{equation}\label{def-fractional-laplacian-1}
	-(-\Delta)^{\frac{\alpha}{2}}f(x)
:=\frac {2^{\alpha} \Gamma (\frac{d+\alpha}{2})}{\pi^{d/2}|\Gamma(-\frac{\alpha}{2})|}
	\int_{\mathbb{R}^d}\frac{f(x+y)+f(x-y)-2f(y)}{\|y\|^{d+\alpha}} \text{d} y,
\end{equation}
where the integral above is defined in the sense of principal value. The constant in front of the integral is chosen to guarantee that for  $\alpha,\beta \in (0,2)$ such that $\alpha +\beta <2$, we have $(-\Delta)^{\frac{\alpha}{2}}(-\Delta)^{\frac{\beta}{2}}f = (-\Delta)^{\frac{(\alpha+\beta)}{2}}f$ for all $f \in C^\infty(\mathbb{T}^d)$. 



\subsubsection{Abstract Wiener Spaces and continuum fractional Laplacians}\label{subsec:AWS}
Given a Fourier multiplier function  $\widehat{K}:\ZZ^d\to\RR_{>0}$ consider 
\begin{equation}\label{norm}
(f,g)_{K,a} := \sum_{\xi\in\ZZ^d\setminus\{0\}} \widehat{K}(\xi) \|\xi\|^{4a} \widehat{f}(\xi) \overline{\widehat{g}(\xi)}.
\end{equation}
We will make the following assumptions on $\widehat{K}$: $\widehat{K}$ is positive, even and real-valued.
It is straightforward to see that \eqref{norm} defines a proper inner product on $C^\infty(\T^d)/\sim$.
Define $H^a_K(\T^d)$ to be the Hilbert space completion of $C^\infty(\T^d)/\sim$ with respect to the norm $(\cdot,\cdot)_{K,a}=\|\cdot\|^2_{K,a}$. 
Recall the definition of abstract Wiener space (see~\S 8.2 in \citet{Str08}).

\begin{definition}\label{aws2}
A triple $(H,B,\mu)$ is called an abstract Wiener space (from now on abbreviated {AWS}) if
\begin{enumerate}
\item $H$ is a Hilbert space with inner product $(\cdot,\cdot)_H$.
\item $B$ is the Banach space completion of $H$ with respect to the measurable norm $\|\cdot\|_B$. Furthermore $B$ is supplied with the Borel $\sigma$-algebra $\Bb$ induced by $\|\cdot\|_B$.
\item $\mu$ is the unique probability measure on $B$ such that for all $\phi\in B^*$ we have $\mu\cdot\phi^{-1} = \mathcal{N}(0,\|\widetilde{\phi}\|_H^2)$, where $\widetilde{\phi}$ is the unique element of $H$ such that $\phi(h)=(\widetilde{\phi},h)_H$ for all $h\in H$.
\end{enumerate}
\end{definition}
In Lemma 3 and 4 in \citet{JanFourier} we proved the following:
\begin{lemma}
Let $\epsilon > d/4 - a$, $\widehat{K}$ a general positive Fourier multiplier and the Sobolev spaces $H^a_K$ w.r.t the  norm \eqref{norm} defined above. Then
the triple $(H^{a}_K,H^{-\epsilon}_K,\mu_{-\epsilon})$ is an AWS. 
\end{lemma}

The measure $\mu_{-\epsilon}$ is the unique Gaussian law on $H^{-\epsilon}_K$ such that the characteristic functional is given by
\begin{equation}\label{chara}
\Phi(f) = \exp \left ( -\frac{1}{2} \| f \|^2_{K,a}\right )
\end{equation}
for a test function $f \in \mathcal{F} =  \left\{C^{\infty}(\T^d); \int_{\T^d} f(x) dx =0 \right \}$. 

\begin{definition}
A Gaussian random field $\Xi^K$ in $H^a_K(\T^d)$ is a collection $\{\langle \Xi^K, f  \rangle; f\in \mathcal{F}\}$ of random variables such that $\mathbb{E}(\langle \Xi^K, f \rangle^2 ) = \| f||^2_{K,a}$.
\end{definition}

Note that we wrote $\Xi^K$ to indicate the dependence of a possible Fourier multiplier and the corresponding Sobolev space $H^{a}_{K}$. We will sometimes just write $\Xi$ or $\Xi^{\star}$ for $\Xi^K$, it will clear from the context what we will mean by that.
 

\subsection{Divisible sandpiles, stabilization and obstacle problem}
In this section we would like to define the divisible sandpile model, state the dichotomy between stabilization and explosion and finally make a connection between the odometer function and solutions to obstacle problems. 

\subsubsection{Divisible sandpiles and odometers}
	
\begin{definition}
A \textit{divisible sandpile configuration} $s$ is a function $s:\mathbb{Z}^d_n \rightarrow \mathbb{R}$. 
\end{definition}

For $x \in \mathbb{Z}^d_n$, if $s(x)\geq 0$, we can interpret $s(x)$ as a \textit{mass} ($s(x)>0$) or \textit{hole} ($s(x)<0$) on the site $x$. If $s(x) > 1$, we call it \textit{unstable} and otherwise \textit{stable}. We then evolve the sandpile according to the following dynamics: unstable vertices will \text{topple} by keeping mass $1$ and redistributing the excess $(s(x)-1)^+$ over the other neighbouring vertices. We will distinguish between two different distribution rules:\\

\noindent
\textbf{Distribution rules:}
\begin{itemize}
\item[(n.n.):] the excess is redistributed according to the transition probabilities of a simple random walk (equally to all nearest neighbours)
\item[(l.r.):] the excess is redistributed according to the transition probabilities $p^{(\alpha)}_n(\cdot)$, see \eqref{def-random-walk-in-zdn}, of a long-range random walk at each discrete time step.
\end{itemize} 

Note that unstable sites in long-range divisible sandpile models distribute mass to \textit{all} vertices (including itself). One could generate a divisible sandpile on  a graph from any random walk defined on it, where on each time step the mass which in each vertex sent to its neighbours is proportional to the transition probabilities. Let us write $(-\Delta)^{\star} \in \left \{ (-\Delta_g),  (-\Delta)_n^{\alpha/2} \right \}$ for denoting a general graph Laplacian which defines the redistribution. We will specify, when needed, the redistribution rule. 

Let $s_t=(s_t(x))_{x\in \mathbb{Z}^d_n}$ denote the sandpile configuration after $t\in \mathbb{N}$ discrete time steps (set $s_0:=s$ the initial configuration).
Most of the times we will use parallel toppling, which we can define via an algorithm in the following way:
\begin{algorithm}\label{alg-divisible-long-range}
Set $t=1$ then run the following loop:
\begin{enumerate}
	\item if $\max_{x \in \mathbb{Z}^d_n} s_{t-1}(x) \le 1$, stop the algorithm;
	\item for all $x \in \mathbb{Z}^d_n$, set $e_{t-1}(x):=(s_{t-1}(x)-1)^+$;
	\item set $s_t(x):=s_{t-1} - (-\Delta)^{\star}_n e_{t-1}(x) $;
	\item increase the value of $t$ by $1$ and go back to step $1$.
\end{enumerate}
\end{algorithm}

Let us illustrate the algorithm in the nearest neighbour case on a simple example.

\begin{figure}[htb]
\includegraphics[scale=1]{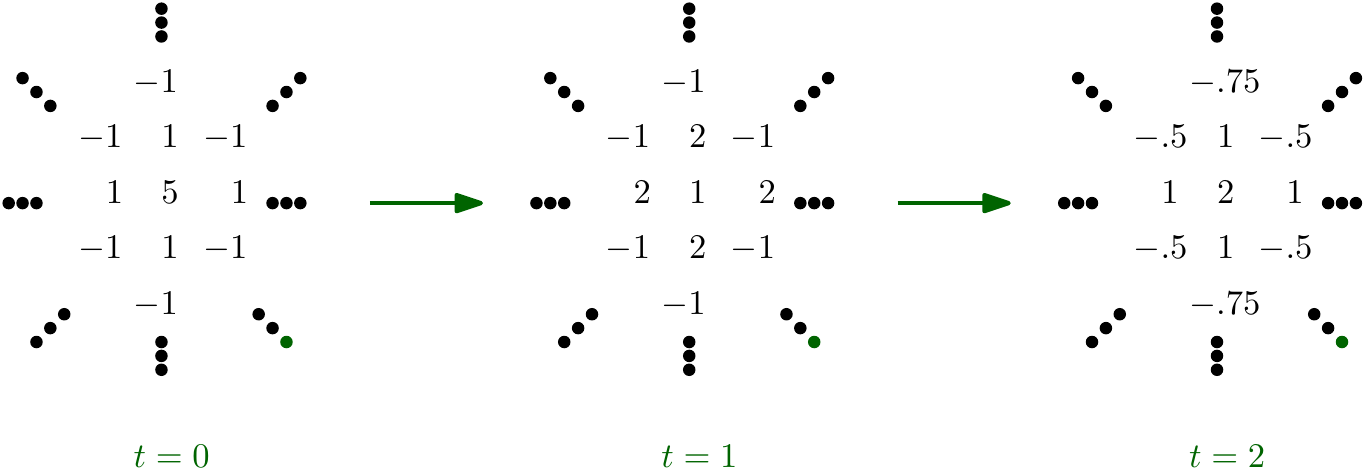}
\caption{Example of a toppling sequence }
\label{fig:example_top}
\end{figure}

We zoom into a local neighbourhood of an initial configuration with the displayed height configuration at $t=0$ in Figure \ref{fig:example_top}. In the next time step we topple $s(x)=5$ and redistributed the excess equally among the nearest neighbours. At $t=2$ we topple in parallel the 4 unstable neighbours $y$ such that $s(y)=2$ of $x$. What we observe is that the holes (``negative heights" ) will slowly fill up and the heights with large mass decrease.

\begin{definition}
The odometer function is equal to $u^{\star}_t: \mathbb{Z}^d_n \rightarrow [0,\infty]$, where $u^{\star}_t(x)$ denotes the total mass emitted from $x$ up to time $t$, that is $u^\star_t(x):= \sum_{i=0}^{t-1} e_{i}(x)$. The superscript $\star \in \{ \text{n.n.,  l.r.}\}$ will combine the nearest neighbour and long-range case.
\end{definition}

Note that analogously to (Section 2 in) \citet{LMPU} we have pointwise:
\begin{equation}\label{equilibrium-equation-long-range-divisible-sandpile}
	s_{t}(x) 
= 
	s(x) - (-\Delta)^{\star} u_{t}(x).
\end{equation}

By construction, the odometer is an increasing function in $t$, call it $u^{\star}_{\infty}:=\lim_{t\rightarrow \infty} u^{\star}_t$. 
The model is Abelian in the sense that the final configuration will not depend on the toppling order.

\subsubsection{Stabilization}

As in Section 1 of \citet{LMPU} and Section 2 of  \citet{LongRange} we define the following dichotomy: either for all $x\in \mathbb{Z}^d_n$ we have \textit{stabilisation}, i.e. $u^{\star}_{\infty}(x)<\infty$ or \textit{explosion}, i.e.  for all $x\in \mathbb{Z}^d_n:$ $u^{\star}_{\infty}(x)=\infty$.
Let us first consider more general toppling procedures, see Definition 2.2 \citet{LongRange} which was adapted from Definition 2.1 of \citet{LMPU}.

\begin{definition}\label{def-top-procedure}
	Let $T \subset [0,\infty)$	be a well-ordered set of toppling times such that
	$0 \in T$, $T$ is closed subset of $[0,\infty)$. A \textit{toppling procedure} is
	a function
	\begin{align*}
		T \times \mathbb{Z}^d_n &\longrightarrow [0,\infty)\\
		(t,x)& \longmapsto u^{\star}_t(x)
	\end{align*}
	such that for all $x \in \mathbb{Z}^d_n$
	\begin{enumerate}
		\item $u^{\star}_0(x) = 0$;
		\item $u^{\star}_{t_1}(x)\le u^{\star}_{t_2}(x)$ for all $t_1 \le t_2$; and
		\item if $t_n \uparrow t$, then $u^{\star}_{t_n}(x)\uparrow u^{\star}_t(x)$.
	\end{enumerate}
\end{definition}

	Given a toppling procedure $u^{\star}$, we say that it is \textit{legal} for the 
	initial configuration $s$ if 
	\[
		u^{\star}_{t}(x)-u^{\star}_{t^-}(x) \le (s_{t^-}(x)-1)^+,
	\]
	where $t^-:= \sup \{ r \in T: r < T \} \in T$ (as $T$ is closed), and \textit{finite} if 
	\[
		u^{\star}_{\infty}(x):= \lim_{t \to \sup T} u^{\star}_t(x) <\infty.
	\]
	The limit is always well defined due the monotonicity of $u^{\star}$.

If $u^{\star}$ is finite, then we have that 
\[
	s_\infty := \lim_{t \to \sup T} s_t
	= s - \lim_{t \to \sup T} (-\Delta)^{\star}u^{\star}_\infty
\]
is well defined and it is equal to $s -(-\Delta)^{\star}u^\star_\infty$.

\begin{definition}\label{def-stab-top-procedure}
	Given a toppling procedure $u^{\star}$, we say that it is \textit{stabilising} for $s$ 
	if $u^{\star}$ is finite and $s_\infty \le 1$ pointwise. We say that $s$ 
	\textit{stabilises} if there exists a stabilising toppling procedure $u^{\star}$ for 
	$s$.
\end{definition}

It is clear that, if $s$ does not stabilize, then $u^{\star}=\infty$.  When $s$ stabilizes, then its odometer $u^{\star}_{\infty}(x)$ is the total amount of mass sent from $x$ to its neighbours in any legal stabilizing toppling procedure for $s$.

The following proposition is also true for general finite connected graphs $V$, we will state it here for the particular case when $V=\Z^d_n$.

\begin{proposition}\label{prop:stab}
Let $s:\Z^d_n\rightarrow \R$ be any initial sandpile configuration satisfying $\sum_{x\in \Z^d_n} s(x)=n^d$ (*). Then $s$ stabilizes to the all 1 configuration and its odometer $u^{\star}$ is the unique function satisfying
\[
\begin{cases}
& s-(-\Delta)^{\star} u^{\star} = 1\\
& \min_{x\in \Z^d_n} u^{\star}(x) =0.
\end{cases}
\]
\end{proposition}
We will sketch the proof which can be found in Lemma 7.1. in \citet{LMPU} or Proposition 3.2 of \citet{ LongRange}. Note that we do not assume anything on the initial configuration, besides  (*), and on the redistribution specified by $\Delta^{\star}$.
\begin{proof}(Sketch.)
First we notice that the the rank of the operator $\Delta^{\star}$ is $rank(\Delta^{\star})=n^d-1$. Then there exists $h:\Z^d_n \rightarrow \R$ such that
\[
-(-\Delta)^{\star} h =s-1.
\]
 Set $f=h-\min h$ then $f\geq 0$ and $s-(-\Delta)^{\star} f=1$ so $f$ stabilizes. By Proposition \ref{prop-least-action} (see next section) the smallest of such solutions is the odometer with minimum equal to 0.
\end{proof}

\subsubsection{Least action principle and the obstacle problem}

The next Proposition 2.5. from \citet{LMPU} is a variational characterization of the odometer function. We rewrote it in terms of our general Laplacian $\Delta^{\star}$.

\begin{proposition}[Least action principle]\label{prop-least-action}
	Let $s \in \mathbb{R}^{\mathbb{Z}^d_n}$, consider
	\[
		\mathcal{F}_{s} :=
		\Big\{
			f: \mathbb{Z}^d_n \longrightarrow \mathbb{R}:
			f \geq 0, s  -(-\Delta)^{\star}f \le 1
		\Big\}.
	\]
	Let $\ell$ be any legal toppling procedure for $s$. We have
	\begin{enumerate}
		\item For all $f \in \mathcal{F}_{s}$ and $x \in \mathbb{Z}^d_n$
			\[
				\ell_\infty \le f;
			\]
		\item for all $u^{\star}$ stabilising toppling procedure for $s$ and $x \in \mathbb{Z}^d_n$
			\[
				\ell_\infty \le u^{\star}_\infty; \text{ and}
			\]
		\item for all $u^{\star}$ legal stabilising toppling procedure for $s$, then for all
		$x \in \mathbb{Z}^d_n$  
		\begin{equation}\label{odoLast}
			u^{\star}_\infty(x) = \inf \{f(x): f \in \mathcal{F}_{s} \} ,
		\end{equation}
		therefore, both $u^{\star}_\infty$ and $s_\infty$ do note depend on the 
		legal choice of the legal  stabilising toppling procedure.
	\end{enumerate}
\end{proposition}

The variational characterisation \eqref{odoLast} has an equivalent formulation summarized in the following lemma:

\begin{lemma}[Lemma 2.2. in \citet{LapGrowth}]\label{odo_lev}
Let $\gamma:\Z^d \rightarrow \R$ satisfy $(-\Delta)^{\star} \gamma = s_0-1$. Then the odometer of $u$ in \eqref{odoLast} is given by
\[
u=v-\gamma
\]
where $v(x) = \inf \{ f(x)|f\geq \gamma; (-\Delta)^{\star}f \leq 0 \}$.
\end{lemma}

The function $\gamma$ is also called the \textit{obstacle} and $v$ (defined above) the \textit{solution to the obstacle problem}.
For a fixed \textit{obstacle surface} $\gamma$ we want to find a graph $f$ which stays above the obstacle and is superharmonic at the same time, meaning in particular that $f$ has no local minima. The solution will be the lowest possible $f$ satisfying the conditions.

\section{Law of the odometer on the discrete torus and discrete fractional Gaussian fields}\label{sec:law}

\subsection{Law of the odometer on the discrete torus}
From Proposition \ref{prop:stab} it is clear that as long as the initial sandpile configuration $s$ satisfies the constraint (*), the sandpile will stabilize to the all 1 configuration. The simplest choice of an initial configuration $(\sigma(x))_{x\in \Z^d_n}$ satisfying (*) is if we take for all $x\in \Z^d_n$:

\begin{equation}\label{ini_s}
s(x) = 1 + \sigma(x) -\frac{1}{n^d} \sum_{y\in \Z^d_n} \sigma(y).
\end{equation}

Note that at this point we do not make any assumptions on the law of $(\sigma(x))_{x\in \Z^d_n}$.
We will discuss two different cases, the i.i.d. case and non-i.i.d. case.\\

\noindent
\textbf{Assumptions on the $\sigma$'s:}
\begin{enumerate}
\item[(A-ind):]  $(\sigma(x))_{x\in \Z^d_n}$ are i.i.d. random variables with $\mathbb{E}(\sigma(x))=0$ and $\var(\sigma(x))=1$.
\item[(A-cor):] $(\sigma(x))_{x\in \Z^d_n}$ are multivariate centered Gaussian variables with stationary covariance $\mathbb{E}(\sigma(x)\sigma(y)) = K_n(x-y)$.
\end{enumerate}
In inverse Fourier transforms of positive, real Fourier multipliers $\widehat{K}_n$ on $\Z^d_n$ correspond to a positive definite function $(x,y) \mapsto K_n (x-y)$ on $\Z^d_n\times \Z^d_n$. 
In Lemma 5 in \citet{JanFourier} we proved an analog of Bochner's theorem, which we will recall for completeness.

\begin{lemma}
The map $(x,y) \mapsto K_n(x,y)=K_n (x-y)$ on $\Z^d_n \times \Z^d_n$ is positive definite, hence a well-defined covariance matrix if and only if the corresponding Fourier multipliers $\widehat{K}_n$ are real-valued, even and positive.
\end{lemma}

Define the Green's function on the torus as

\begin{equation}
	g^{\star}(x,y)
=
	\frac{1}{n^d} \sum_{z \in \mathbb{Z}^d_n }g^{\star}_z(x,y) 
\end{equation}
and $g^{\star}_z(x,y)=\mathbb{E}_x[\sum_{t=0}^{\tau_z-1} \1_{\{X^{\star}_t=y \}}]$ is the expected number of visits to $y$ by a random walk $(X^{\star}_t)_{t\in \N}$ starting from $x$ before being killed at $z$. Recall,  we write $(X^{\star}_t)_{t\in \N}$ for $\star \in \{ n.n., l.r. \}$ so if $\star=n.n.$ then we mean the simple random random walk and if $\star=l.r.$ we will consider long-range random walk with transition probabilities specified by \eqref{def-random-walk-in-zdn}. 

The following Proposition \ref{prop:dislaw} summarizes the laws of the odometers for different distribution rules and initial conditions. It is a combination of Proposition 1.3 from \citet{LMPU}, Proposition 4 in \citet{CHR17}, Proposition 3.2 in \citet{LongRange} and Lemma 6 in \citet{JanFourier}.

\begin{proposition}\label{prop:dislaw}
Let $s$ be an initial divisible sandpile configuration satisfying \eqref{ini_s}, namely
\[
s(x) = 1 + \sigma(x) - \frac{1}{n^d}\sum_{y\in \Z^d_n} \sigma(y)
\]
and the collection of random variables $(\sigma(x))_{x\in \Z^d_n}$ satisfying (A-ind) or (A-cor).  Then the law of the odometer $u^{\star}$ is equal to  

\begin{equation}\label{char-field-long-range-divisible-eq-1}
	\{u^{\star}_\infty(x)\}_{x \in \mathbb{Z}^d_n}
\sim 
	\Big\{
		\eta^{\star}(x)- \min_{z \in \mathbb{Z}^d_n}\eta^{\star}(z)
	\Big\}_{x \in \mathbb{Z}^d_n },
\end{equation}
where $\eta^{\star}(x)=\sum_{y\in \Z^d_n} g^{\star}(x,y)(s(y)-1)$ and in particular has covariances given by: 
\begin{equation}\label{cov}
\mathbb{E}(\eta^{\star}(x)\eta^{\star}(y)) = \begin{cases} 
\sum_{z\in \Z^d_n} g^{n.n.}(x,z)g^{n.n.}(z,y) & \text{ for } \star=n.n., (A-ind) \\
\sum_{z\in \Z^d_n} g^{l.r.}(x,z)g^{l.r.}(z,y) & \text{ for } \star=l.r., (A-ind) \\
\sum_{z,z'\in \Z^d_n} K_n(z-z')g^{n.n.}(x,z)g^{n.n.}(z',y) & \text{ for } \star=n.n., (A-cor) 
\end{cases}
\end{equation}
for all $x,y\in \Z^d_n$.
\end{proposition}
The proof is a direct consequence of Proposition \ref{prop:stab}. Note that the covariance function $f(x,y):=\mathbb{E}(\eta^{\star}(x)\eta^{\star}(y))$ solves the following discrete PDE (up to constants)
\[
(-\Delta^{\star}_n)^2 f(x,y) = \delta_x(y) -\frac{1}{n^d}.
\]
Of particular interest are initial $\sigma$'s which are Gaussian. Then the distribution of $u^{\star}$ is completely determined by the covariance $(f(x,y))_{x,y\in \Z^d_n}$. In fact, assume that $\star=n.n$ and $\sigma$'s are standard normals, then we retrieve the discrete bi-Laplacian field
\[
(\eta(x))_{x\in \Z^d_n} \sim N(0, (f(x,y))_{x,y})
\]
where $f(x,y)$ solves the discrete bi-Laplacian equation $(-\Delta_g)^2 f(x,y) = \delta_x(y) -\frac{1}{n^d}$. Let us look at some examples for different laws of the odometer $u^{\star}$ depending on the diverse  assumptions.\\

\noindent
\textbf{Examples:}

\textbf{Case 1: nearest neighbour distribution and $\sigma$'s are i.i.d. standard normals}\\
For the nearest neighbour case and under assumption (A-ind) when the $\sigma$'s are independent standard normals, the odometer interface is depicted in Figure \ref{uncor}.
\begin{figure}[htb]
\includegraphics[scale=0.3]{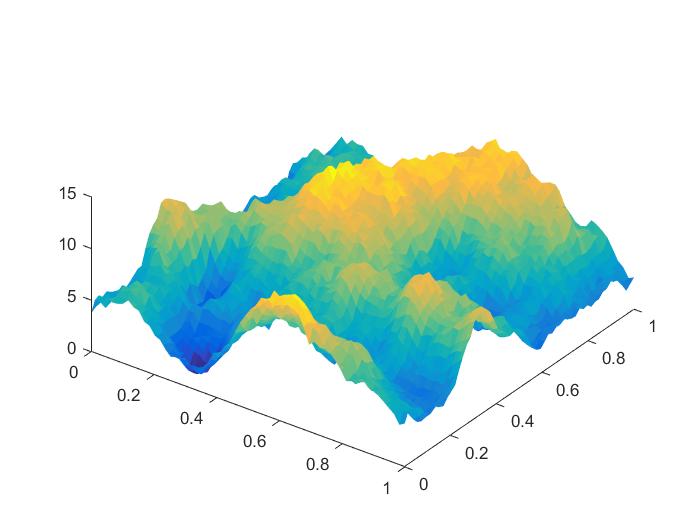}
\caption{Realization of the odometer for $u^{n.n.}$ and $\sigma$'s are i.i.d. $\sigma \sim N(0,1)$}
\label{uncor}
\end{figure}

\textbf{Case 2: nearest neighbour distribution and $\sigma$'s are correlated}\\
Let $\sigma$'s is $(\sigma(x))_{x\in \mathbb{Z}^d_n} \sim \mathcal N(0,\, K_n)$ when
the covariance matrix $K_n$ is polynomially decaying. Consider for example the summable kernel ($K\in \ell^1(\Z^d)$):
\[
K_n^{\pm}(x,y) = \begin{cases} 7 & \text{ if } x=y \\  \pm \| x-y\|^{-3} & \text{ otherwise }.\end{cases}
\]
The corresponding realizations of the odometer function are indicated in Figures~\ref{Fig:PosK}-\ref{Fig:NegK}. 
\begin{figure}[h]
\centering
 \centering
    \begin{minipage}{.5\textwidth}
\includegraphics[scale=1]{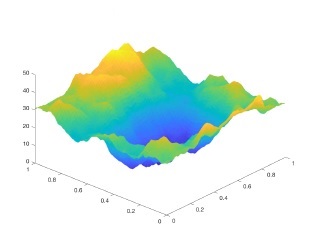}
\caption{Realization of the odometer associated to $K_n^+$.}\label{Fig:PosK}
\end{minipage}%
    \begin{minipage}{.5\textwidth}
\centering
\includegraphics[scale=1]{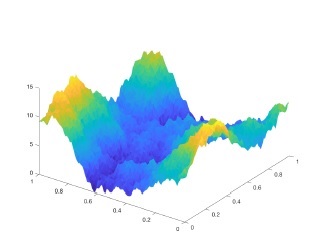}
\caption{Realization of the odometer associated to $K^-_n$.}\label{Fig:NegK}
\end{minipage}
\end{figure}

A combination of the two correlated interface is presented in Figure \ref{fig_comp}.
\begin{figure}[htb]
\includegraphics[scale=0.5]{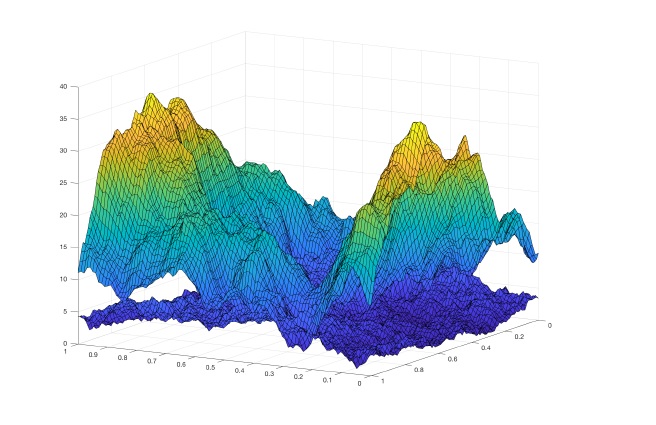}
\caption{Combination of the odometer interfaces associated to $K^+_n$ and $K^-_n$.}
\label{fig_comp}
\end{figure}
Observe that the realizations of the odometer function w.r.t. correlated initial $\sigma's$ look similar to the uncorrelated case. The positive or negative correlations play a role for maxima of the interface and how much it fluctuates. Positive correlations increase the position of the maxima and negative decrease them. The last Figure \ref{smoth} corresponds to choose $K_n(x,y)=C\|x- y\|^{-1}$, so $K_n\notin \ell^1(\Z^d)$. We get a smooth interface corresponding to the field $\Delta^{-(1+1/4)}W$ in the continuum.

\begin{figure}[htb]
\includegraphics[scale=0.3]{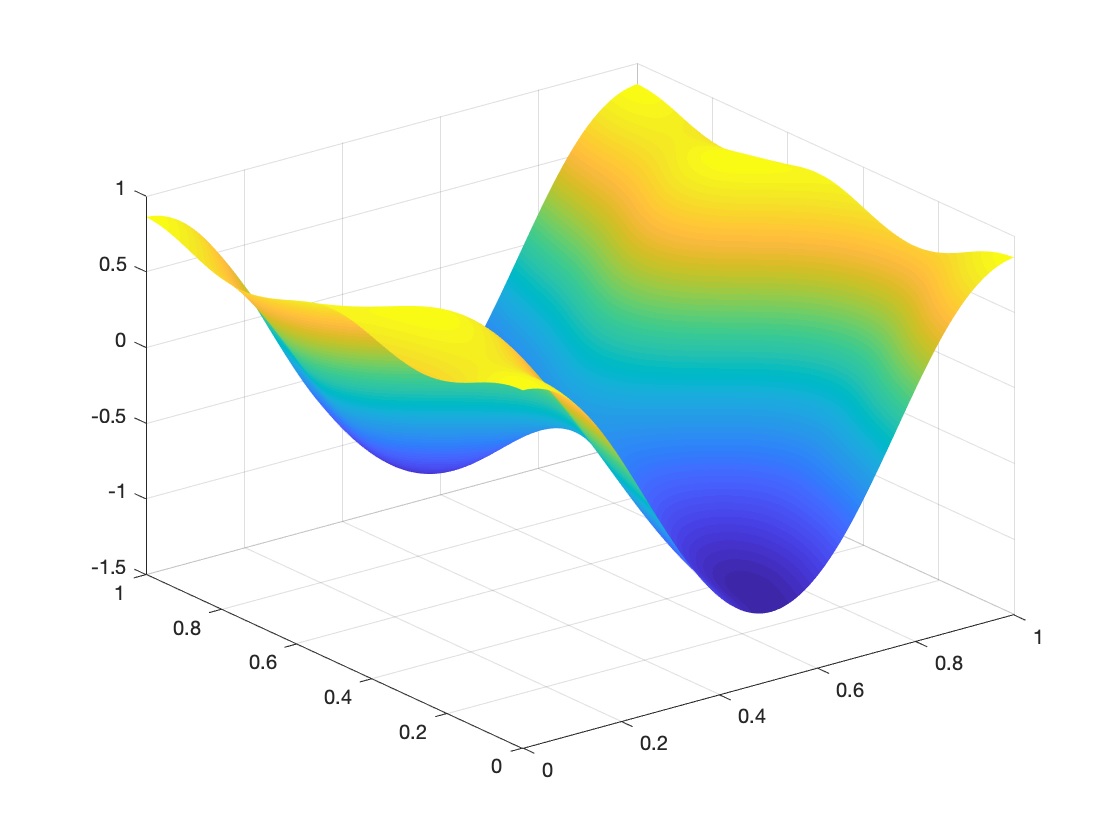}
\caption{Odometer for non-summable correlations}
\label{smoth}
\end{figure}


\textbf{Case 3: long-range distribution and $\sigma$'s are i.i.d. standard normals}\\
The third example concerns initial $\sigma$'s which are i.i.d. standard Gaussians and a long-range type redistribution w.r.t. a parameter $\alpha$.

\begin{figure}[h]
\centering
 \centering
    \begin{minipage}{.5\textwidth}
    \includegraphics[width=0.8\linewidth]{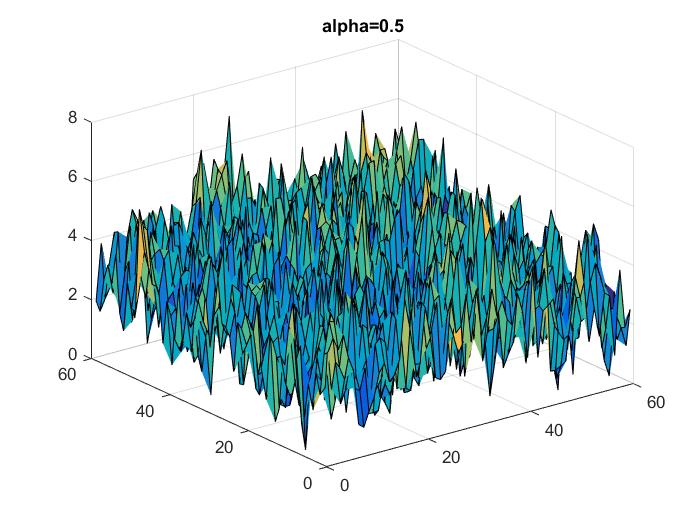} 
\caption{Realization of the odometer $u^{l.r.}$ associated to $\alpha=0.5$.}
\end{minipage}%
    \begin{minipage}{.5\textwidth}
\centering
    \includegraphics[width=0.8\linewidth]{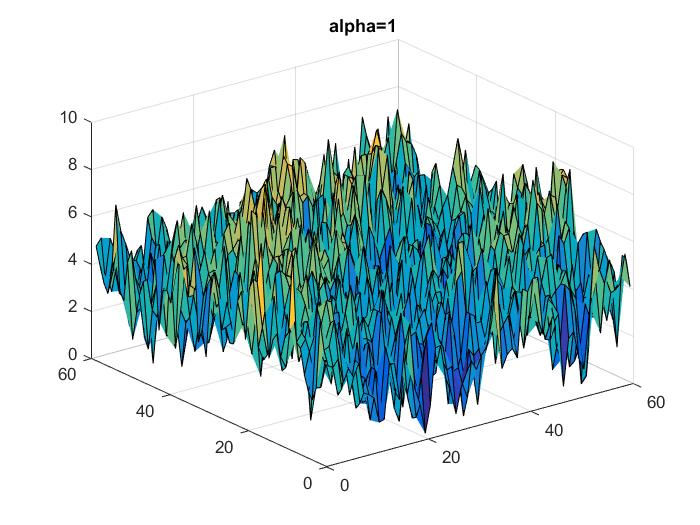} 
\caption{Realization of the odometer $u^{l.r.}$ associated to $\alpha=1$.}
\end{minipage}
\end{figure}

\begin{figure}[h]
\centering
 \centering
    \begin{minipage}{.5\textwidth}
    \includegraphics[width=0.8\linewidth]{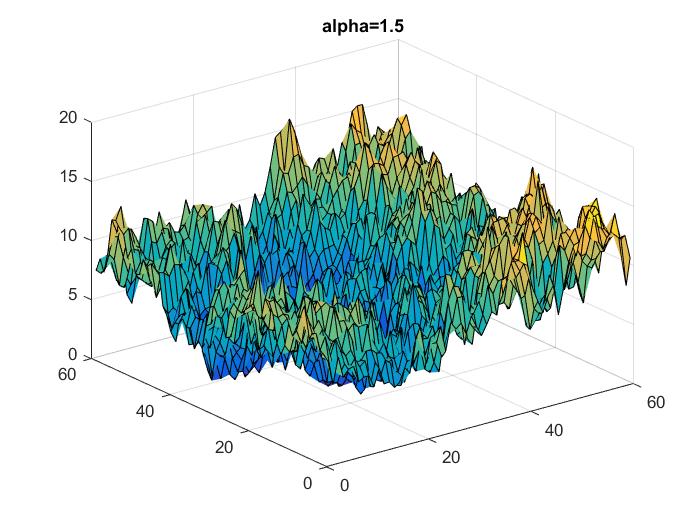} 
\caption{Realization of the odometer $u^{l.r.}$ associated to $\alpha=1.5$.}
\end{minipage}%
    \begin{minipage}{.5\textwidth}
\centering
    \includegraphics[width=0.8\linewidth]{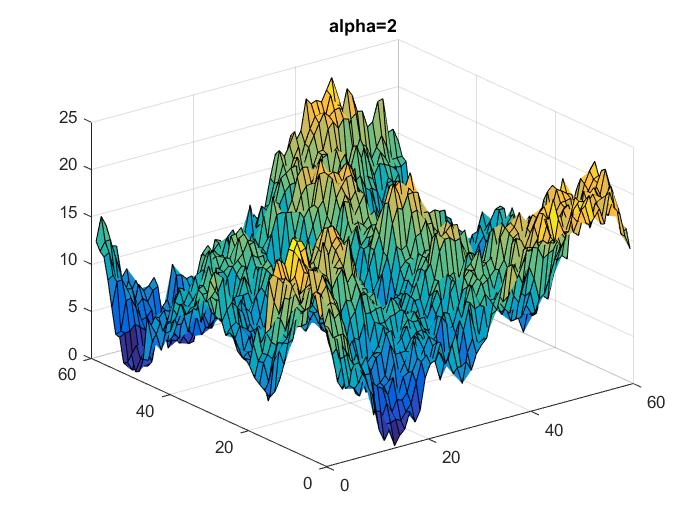} 
\caption{Realization of the odometer $u^{l.r.}$ associated to $\alpha=2$.}
\end{minipage}
\end{figure}

Observe that the lower $\alpha$ the ``rougher" the interface looks like.
\subsection{Mean behaviour of the odometer, extrema and fluctuations}

As described in Section 1.3 of \citet{LMPU}, the expected odometer can be seen as a \textit{measure of how difficult it is to stabilize}. In other words, it gives the information on how much mass is emitted in average from each site depending on dimension.  
Let $\sigma(x) \sim N(0,1)$ being i.i.d. and satisfying $\sum_{x\in \Z^d_n}s(x)=n^d$ and we will now consider general $\Delta^{\star}$. Note that from Proposition \ref{prop:dislaw}
we obtain by invariance:
\[
\mathbb{E}\left(u^{\star}_{\infty}(o) \right) = \mathbb{E}\left(\max_{x\in \Z^d_n} \eta^{\star}(x) \right).
\]
The order of the mean odometer will give the order of the maxima of the Gaussian field.

We combine Theorem 1.2 and Proposition 8.3, 8.8 from \citet{LMPU} in the following Theorem \ref{max_odo_nn}.  The proof replies on the fact that in law the odometer is equal to a discrete bi-Laplacian field shifted to have minimum value 0.

\begin{theorem}\label{max_odo_nn}
There exists a constant $C_d>0$ such that 
\[
C_d^{-1} \phi_d(n) \leq \mathbb{E} (u_{\infty}^{n.n.}(x)) \leq C_d \phi_d(n)
\]
where for $n\geq 2$, the function $\phi_d$ is defined as
\[
\phi_d(n) =\begin{cases}
n^{2- d/2}, & d<4 \\
\log(n), & d=4 \\
(\log(n))^{1/2}, & d\geq 5.   
\end{cases}
\]
Moreover $\mathbb{E}\left(\max_{x\in \Z^d_n} \eta^{n.n.}(x) \right) \asymp \phi_d(n)$.
\end{theorem}

A combination of Proposition 8.2 and 8.7 in \citet{LMPU} yields the following proposition for the variance of the $\eta$- configuration, (see also Proposition \ref{prop:dislaw} for the definition of $\eta$) .

\begin{proposition}
There exist $c_d,C_d>0$ such that, if $r=\|x\|_2$,
\[
c_d \psi_d(n,r) \leq \mathbb{E}((\eta^{n.n.}(x)-\eta^{n.n.}(o))^2) \leq C_d \psi_d(n,r)
\]
where
\[
\psi_d(n,r) = \begin{cases}
n r^2, & d=1\\
r^2 \log(n/r), & d=2\\
r, & d=3 \\
\log(1+r) & d=4 \\
1, & d\geq 5. 
\end{cases}
\]
\end{proposition}

If we now allow to spread mass relative to the distance of the points, hence according to a long-range random walk, how fast does the sandpile stabilize?
It is intuitively clear that, if mass will spread to far at each toppling, then the stabilization should be faster than in the previous case. Indeed it is the case, namely let us state Theorem 3.4. from \citet{LongRange}

\begin{theorem}\label{max_odo}
Let $\alpha \in (0,\infty)\setminus  \{2\}$, $\sigma(x) \sim N(0,1)$ i.i.d., $\star=l.r.$ and $\gamma=\min\{2,\alpha \}$. Then there exist $C_d$ such that
\[
C^{-1}_d \Phi_d(n) \leq \mathbb{E}(u_{\infty}^{l.r.}(x)) \leq C_d \Phi_d(n)
\]
where
\[
\Phi_d(n) = \begin{cases}
n^{\gamma - d/2}, & \gamma > d/2\\
\log(n), & \gamma=d/2\\
(\log(n))^{1/2}, & \gamma>d/2. \\
\end{cases}
\]
\end{theorem}

Note the interesting point that $\gamma\leq 2$ so we will never have the correspondence to models with critical dimension higher than 4.
The next Proposition 4.6 from \citet{LongRange} will give us the order of the variance of the $\eta$'s. It is the crucial computation to get Theorem \ref{max_odo}.

\begin{proposition}\label{prop-upper-bounds-gaussian-distance}
For $\alpha \in (0,2)$ there is a constant $C = C_{d,\alpha}>0$ such that for all $n \geq 1$ and all $x \in \mathbb{Z}^d_n$,
\[
    \mathbb{E}[(\eta^{l.r.}(0) - \eta^{l.r.}(x))^2] 
\le
    C \Psi_{d,\alpha(n,\|x\|)},
\]
where 
\begin{equation}\label{def-asymp-mean-u}
    \Psi_{d,\alpha}(n,r):=
    \begin{cases}
        n^{2\alpha - d-2} r^2        ,& \text{ if } \alpha > \frac{d}{2}+1\\ 
        \log(\frac{n}{r}) r^2    ,& \text{ if } \alpha = \frac{d}{2}+1\\
        r^{2\alpha -d}            ,& \text{ if } \alpha \in  (\frac{d}{2},\frac{d}{2}+1)\\
        \log(r)                    ,& \text{ if } \alpha = \frac{d}{2}\\
        1                        ,& \text{ if } \alpha > \frac{d}{2}\\
    \end{cases}.
\end{equation}
\end{proposition}

\begin{proof}
The proof goes via noticing that 
\begin{align*}
    \mathbb{E}[(\eta^{l.r.}(0) - \eta^{l.r.}(x))^2]
&= 
    \sum_{ w \in \mathbb{Z}^d_n} (g^\alpha(w,0)-g^\alpha(w,x))^2\\
& = 
    \frac{1}{n^d}\sum_{ w \in \mathbb{Z}^d_n\backslash \{0\}} \frac{\sin^2(\frac{\pi x\cdot w}{n})}{(\lambda^{(\alpha,n)}_w)^2} \\
& = \int_{\R^d} G_{n,d,\alpha,x}(y) dy \\
& = \int_{\R^d}  \sum_{ w \in \mathbb{Z}^d_n\backslash \{0\}} \frac{\sin^2(\frac{\pi x\cdot w}{n})}{(\lambda^{(\alpha,n)}_w)^2} 1_{B(\frac{w}{n},\frac{1}{2n})}(y) dy.
\end{align*}
 Lemma 4.7 from \citet{LongRange} roughly states that $\lambda_{w}^{(\alpha, n)} \asymp \| w/n \|^{\gamma}$ where recall that $\gamma=\min \{2,\alpha \}$ and extra logarithmic corrections for $\gamma=2$. The proof goes via  showing that for some $C>0$
\[
C^{-1} \sum_{ w \in \mathbb{Z}^d_n\backslash \{0\}} \frac{\sin^2(\frac{\pi x\cdot w}{n})}{ \|w\|^{2\gamma}} \1_{B(\frac{w}{n},\frac{1}{2n})}(y) \leq G_{n,d,\alpha,x}(y) \leq C \sum_{ w \in \mathbb{Z}^d_n\backslash \{0\}} \frac{\sin^2(\frac{\pi x\cdot w}{n})}{ \|w\|^{2\gamma}} \1_{B(\frac{w}{n},\frac{1}{2n})}(y).
\] 
The most involved technical part is to get the proper upper and lower bounds for the terms 
\[
H_{n,d,\alpha,x}(y)=\sum_{ w \in \mathbb{Z}^d_n\backslash \{0\}} \frac{\sin^2(\frac{\pi x\cdot w}{n})}{ \|w\|^{2\gamma}} \1_{B(\frac{w}{n},\frac{1}{2n})}(y) 
\]
depending on the dimension and parameter $\alpha$ Lemmas 4.8-4.15 in \citet{LongRange}. Comparisons with proper Riemann intergrals are used for that.
\end{proof}

\begin{proof}[Proof of Theorem \ref{max_odo}]
The main ingredients are (similarly to the proof of Theorem \ref{max_odo_nn} in \citet{LMPU}) 
 Dudley's bound Proposition 1.2.1 in \citet{Tal05} and the Majorizing  Measure  Theorem 2.1.1. in \citet{Tal05}. 
The idea is to study the mean of the extremes of a centered Gaussian field $\{\eta(x)\}_{x \in T}$ for some set of indexes through the metric in $T$ induced by 
\begin{align}\label{def-gaussian-metric}
\nonumber
        d_\eta : T \times T &\longrightarrow \mathbb{R} 
\\ 
        (x,y)& \longmapsto \mathbb{E}[(\eta(x) -\eta(y))^2]^{\frac{1}{2} }.
\end{align}
``Good bounds" for $d_{\eta}(x,y)$ (which will come from Propostition \ref{prop-upper-bounds-gaussian-distance}) will imply ``good bounds" for $\mathbb{E}\left (\max_{x \in T} (\eta^{l.r.}(x) )\right )$ using similar computations as in \citet{LMPU}.
\end{proof}

\subsection{Discrete fractional Gaussian fields}

Discrete Gaussian Free Fields (DGFF) on the torus  were studied for example in \citet{ABDGFF} and  \citet{AleMani}. 
In \citet{ABDGFF} the author studied level-sets percolation of the  Gaussian Free Field (GFF) on the discretized $d$-dimensional torus $\Z^d_n$ for $d\geq 3$. He  approximates the GFF on the torus by a GFF on the lattice $\Z^d$. The authors in \citet{AleMani} construct a DGFF on a compact manifold. To the author's knowledge the first time the discrete fractional Gaussian field via long-range random walks on the torus was defined, is in \citet{LongRange}. It is consistent with the definition on regular bounded domains $D\subset \R^d$ which can be found in \citet{LSSW}. Note that the authors in \citet{LSSW} (see section 12.2) define it for $\alpha \in (0,2)$ whereas in \citet{LongRange} it is defined for all $\alpha >0$ in equation 2.2.

\begin{definition}
For $\delta >0$ define $V^{\delta}:=\delta \Z^d \cap D$. Then the zero-boundary discrete fractional Gaussian field (DFGF) $h^{\delta}$ is a multivariate Gaussian function with density at $f\in \R^{V^{\delta}}$ proportional to
\[
\exp \left ( -\frac{1}{2}\sum_{(x,y)\in (\delta \Z)^2, x\neq y }  \frac{|f(x)-f(y)|}{|x-y|^{d+2\alpha}} \delta^d\right ).
\]
\end{definition}

Then $h^{\delta}$ can be interpreted as a linear functional on $C^{\infty}_c(D)$ by setting
\[
(h^{\delta}, f) =\sum_{x\in V^{\delta}} h^{\delta}(x) f(x) \delta^d
\]
for all test functions $f \in C^{\infty}_c(D)$.

\section{Scaling limits of the odometer and continuum fractional Gaussian fields}\label{sec:scale}

\subsection{Continuum fractional fields}

Let $\alpha \in \R$, $D\subset \R^d$ some regular domain and $W$ spatial white noise on $\R^d$. Recall that $W$ is a spatial white noise on $\R^d$ if for all Schwartz functions $f\in \mathcal{S}(\R^d)$ 
we have that $\langle W,f\rangle \sim N(0,\|f\|^2_{L^2(\R^d)})$.
We adapt the formal notation of \citet{LSSW} to describe
the family of fractional Gaussian fields in the continuum by
\begin{equation}\label{FGF}
FGF_{\alpha}(D) :=\Delta^{-\alpha/2} W.
\end{equation}
Special cases include the GFF for $\alpha=1$ or the bi-Laplacian model for $\alpha=2$.
Depending on the parameters $\alpha,\,d$ the field $\Delta^{-\alpha/2} W$ will be either a random distribution or a random continuous function. A random distribution (contrary to a function) cannot be defined pointwise but instead tested against suitable test functions.  More precisely, if the Hurst parameter $H$ of the $FGF_{\alpha}(D)$  field
\[
 H:=\alpha-\frac{d}{2}
\]
is strictly negative, then the limit is a random distribution while for $H\in (k,\,k+1)$, $k\in \NN\cup\{0\}$, the field is a $(k-1)$-differentiable function (\citet{LSSW} presented the results  for domains with zero boundary conditions).

\begin{definition}
For $d>2\alpha$ and $\alpha \in (0,2)$ we define the a fractional Gaussian field $\Xi^{\alpha}$ as the unique distribution on $(C^{\infty}_c(\R^d))^*$ such that for all test functions $f\in C^{\infty}_c(\R^d)$,  the random variable  $\langle \Xi^{\alpha}, f\rangle$ is a centered Gaussian variable with variance
 \[
\mathbb{E}(\langle \Xi^{\alpha}, f\rangle^2) = \int_{\R^d}\int_{\R^d} f(x)f(y)\|x-y\|^{2\alpha -d} dxdy.
\]
\end{definition}

\subsection{Scaling limits of the odometer}
Let us define the formal field for $x\in \T^d:$
\begin{equation}\label{xi}
\Xi^{\star}_n(x) := \sum_{z\in \T^d_n} u^{\star}_{\infty} (z\cdot n) \1_{B(z,1/2n)}(x).
\end{equation}
It means that for every point in the continuum torus $x\in \T^d$ we assign the value of the odometer on the discretized torus which is in a ball of radius 1/2n around $z$ containing $x$.  The scaling $a_n^{\star}$ of the formal field will depend on the initial configuration and redistribution rule. For pedagogical reasons we will write the results in 2 different theorems. The first Theorem \ref{theo_nn} is proven in \citet{CHR17} (see Theorem 1+2) and \citet{JanFourier} (refer to Theorem 1) and the second Theorem \ref{theo_lr} is proven in \citet{LongRange} in Theorem 3.5.

\begin{theorem}[Scaling limit for $\star=n.n$]\label{theo_nn}
Let the initial sandpile configuration $s$ be defined as in \eqref{ini_s} with $\sigma$'s satisfying (A-ind) or (A-cor). Define the scaling $a_n^{n.n.}$ as
\[
a_n^{n.n.} =\begin{cases} 4\pi^2 n^{\frac{d-4}{2}} & \text{ if }  (A-ind) \\ 
 4\pi^2 n^{-2} & \text{ if }  (A-cor) 
\end{cases}
\]
Then the formal field defined in \eqref{xi} converges in law towards
\[
a^{n.n.}_n \Xi^{n.n}_n \rightarrow \begin{cases} 
\Delta^{-1}W & \text{ if }  (A-ind)\\ 
\Delta^{- (1+\delta)}W & \text{ if }  (A-cor), \delta >0\\ 
\end{cases}
\]
on $\T^d$. The convergence holds in the Sobolev space $H^K_{-\epsilon}(\T^d)$ with topology induced by the norm $\|\cdot \|_{-\epsilon,K}$ for $\epsilon > \max\{d/4+1,d/2\}$.
\end{theorem}

Remark that for the (A-ind) case the Fourier multiplier $\widehat{K}$ is constant equal to 1 and recall that we use the formal notation \eqref{FGF} for the limiting fields $\Delta^{-1}W$ resp. $\Delta^{- (1+\delta)}W$.

The fields obtained in the previous theorem have regularity properties of bi-Laplacian models or smoother interface models for $\Delta^{- (1+\delta)}W $. Let us remark that the main Theorem 1 proved in \citet{JanFourier} is even stronger, more precisely, the Gaussian limiting fields have variance given by
\[
\mathbb{E}(\langle \Xi , f \rangle^2) = \sum_{z\in \Z^d \setminus \{o\}} \widehat{K}(z)\|z\|^{-4} |\widehat{f}(z)|^2
\]
and depending on a Fourier multiplier which is related to the covariance kernel of the initial distribution of $\sigma$. Choosing $\mathbb{E}(\sigma(x)\sigma(y))=K_n(x-y)$ such that $\widehat{K}(z) = \|z\|^{-4\delta}$ on $\Z^d\setminus 0$ provides the field $\Delta^{- (1+\delta)}W $ for $\delta>0$. If the initial covariance kernel is summable on $\Z^d$ then $\widehat{K}(z)=Const$ and we obtain the bi-Laplacian model.

The following theorem states the scaling limit result for the long-range case. Note that in particular we can obtain a whole class of limiting fields with different regularity properties including the GFF on the torus. The odometer approach presents an alternative way to construct for example a Gaussian Free Field on the torus.

\begin{theorem}[Scaling limit for $\star=l.r$]\label{theo_lr}
Let the initial sandpile configuration $s$ be defined as in \eqref{ini_s} with $\sigma$'s satisfying (A-ind). Define the scaling $a_n^{l.r.}$ as
\[
a_n^{l.r.} = \begin{cases}
n^{\frac{d-2\alpha}{2}} & \text{ if } \alpha < 2 \\
n^{\frac{d-2\alpha}{2}}\log(n) & \text{ if } \alpha = 2 \\
n^{\frac{d-4}{2}} & \text{ if } \alpha > 2.
\end{cases}
\]
Then the formal field defined in \eqref{xi} converges in law to
\[
a^{l.r.}_n \Xi^{r.r}_n \rightarrow \Delta^{-\gamma/2}W.
\]
The convergence holds in the Sobolev space $H_{-\epsilon}(\T^d)$ with topology induced by the norm $\|\cdot \|_{-\epsilon}$ for $\epsilon > \max\{d/4+ \gamma/2,d/2\}$ and $\gamma=\min\{2,\alpha\}$.
\end{theorem}

The proofs involve showing finite dimensional convergence and tightness. For ensuring tightness we will need to assume that $\epsilon > d/2$. On the other side for the AWS to be well-defined we will need that for the nearest neighbour case we have that $\epsilon > 1+d/4$ and in the long-range case $\epsilon > \gamma/2 + d/4$.

The kernels of the operators will be specified in Theorem \ref{kernel} below.
It is a combination of Corollary 1 in \citet{LongRange} and Theorem 3 of \citet{CHR17} which follows from the proofs of the scaling limit results in Theorem's \ref{theo_lr}-\ref{theo_nn}. We added a result for the kernel of the operator in the correlated case (A-cor) which is not specified in \citet{JanFourier} but follows from analogous computations as before in \citet{CHR17}. We will distinguish between two cases, below the critical dimension $d^{\star}_c$ which is equal to $d_c^{nn}=4$ for $\star=n.n.$ and $d^{lr}_c=2\gamma$ for $\star=l.r.$ and above. Note that the kernel at the critical dimension $d=4$ for the nearest neighbour case and $d=2\gamma$ for the long-range case is still an open question. We expect some logarithmic behaviour to play a role.  

\begin{theorem}[Kernel of the fractional operator for  $d>d_c^{\star}$]\label{kernel}

\begin{enumerate}
\item[Case $\star=n.n.$, (A-ind):] The kernel $\mathbb{K}$ of the bi-Laplacian operator on the torus is a $L^1(\R^d)$-function  and is equal to
\[
\mathbb{K}(x,y) = c_d \|x-y\|^{4-d} + h_d(x-y)
\]   
where $h_d \in C^{\infty}(\R^d)$ and $c_d>0$ some constant depending on $d$.
\item[Case $\star=n.n.$, (A-cor):] Let us consider initial correlated Gaussian's $\sigma$ such that $\widehat{K}(\cdot)=\| \cdot \|^{-4\delta}$ with $\delta>0$. The kernel  $\mathbb{K}$  of the fractional field on the torus is a $L^1(\R^d)$-function such that
\[
\mathbb{K}(x,y) = c_{d,\delta} \|x-y\|^{4(1+\delta)-d} + h_{d,\delta}(x-y)
\]
for $h_{d,\delta} \in C^{\infty}(\R^d)$ and $c_{d,\delta}>0$ some constant depending on $d$ and $\delta$.
\item[Case $\star=l.r.$, (A-ind):] Let $\alpha \in (0,2)$. The kernel  $\mathbb{K}$  of the fractional field on the torus is a $L^1(\R^d)$ function such that
\[
\mathbb{K}(x,y) = c_{d,\delta} \|x-y\|^{2\alpha-d} + h_{d,\alpha}(x-y)
\]
for $h_{d,\alpha} \in C^{\infty}(\R^d)$ and $c_{d,\alpha}>0$ some constant depending on $d$ and $\alpha$.
\end{enumerate}
\end{theorem}

\begin{theorem}[Kernel of the fractional operator for  $d<d_c^{\star}$]
\begin{enumerate}
\item[Case $\star=n.n.$, (A-ind):] For $x, y \in \T^d:$
$$\mathbb{K}(x,y) =  \sum_{v\in \Z^d \setminus \{o\}} \frac{e^{2\pi i (x-y)v}}{\|v\|^4}$$. 
\item[Case $\star=n.n.$, (A-cor):] Let us consider initial correlated Gaussian's $\sigma$ such that $\widehat{K}(\cdot)=\| \cdot \|^{-4\delta}$ with $\delta>0$. Then 
\[
\mathbb{K}(x,y) = \sum_{v\in \Z^d \setminus \{o\}} \frac{e^{2\pi i (x-y)v}}{\|v\|^{4(1+\delta)}}. 
\]
\item[Case $\star=l.r.$, (A-ind):] Let $\alpha \in (0,2)$. Then 
\[
\mathbb{K}(x,y) = \sum_{v\in \Z^d \setminus \{o\}} \frac{e^{2\pi i (x-y)v}}{\|v\|^{2\alpha}}. 
\]

\end{enumerate}
\end{theorem}

The proofs are based on rewriting the variance of the field as an double integral and identifying the kernel. A priori the kernel can be written as a sum in any dimension (compare with the previous theorem). The fact that in the supercritical dimensions $d>d^{\star}_c$ power-law decay of correlations appears is due to the existence of infinite volume Gibbs measures in $\R^d$ of the corresponding bi-Laplacian field or fractional field. 
In the sub-critical dimensions there is only an implicit formula for the kernel.

\subsection{Main ideas of the scaling limit proofs}

Let us describe the main ingredients and ideas of the proofs of Theorems \ref{theo_lr}-\ref{theo_nn}, based on \citet{CHR17, LongRange} and \citet{JanFourier}. 
Recall that we proved in Proposition \ref{prop:dislaw} that the odometer on the discrete torus is distributed as
\begin{equation*}
	\{u^{\star}_\infty(x)\}_{x \in \mathbb{Z}^d_n}
\sim 
	\Big\{
		\eta^{\star}(x)- \min_{z \in \mathbb{Z}^d_n}\eta^{\star}(z)
	\Big\}_{x \in \mathbb{Z}^d_n },
\end{equation*}
where  $\eta^{\star}(x)=\sum_{y\in \Z^d_n} g^{\star}(x,y)(s(y)-1)$ has mean 0 and covariance specified in \eqref{cov}. Since we have mean 0 test functions $f$, note the important simplification: 
\[
\begin{split}
\langle \Xi^{\star}_n , f\rangle & = \int_{\T^d} \left (\sum_{z\in \T^d_n} \1_{B(z,1/2n) }(x)u^{\star}_{\infty}(nz) \right ) f(x) dx \\
& = \int_{\T^d} \left (\sum_{z\in \T^d_n} \1_{B(z,1/2n) }(x)\eta^{\star}_{\infty}(nz) \right ) f(x) dx -  [\min_{y\in \Z^d_n} \eta^{\star}(y)] \int_{\T^d} \left (\sum_{z\in \T^d_n} \1_{B(z,1/2n) }(x)  \right ) f(x) dx\\
& = \int_{\T^d} \left (\sum_{z\in \T^d_n} \1_{B(z,1/2n) }(x)\eta^{\star}_{\infty}(nz) \right ) f(x) dx \\
& = \int_{\T^d} \left (\sum_{z\in \T^d_n} \1_{B(z,1/2n) }(x) \left [ \sum_{w\in \Z^d_n} g^{\star}(nz,w)\sigma(w) \right] \right ) f(x) dx
\end{split}
\]
We need to impose mean zero on the test function to get rid of the 0 eigenvalue of the (fractional) Laplacian(s) in order to be able to invert them. Indeed,
\[
\mathbb{E}(\eta_{\infty}^{\star}(x)\eta_{\infty}^{\star}(y)) = const + \sum_{z\in \mathbb{Z}^d_n\setminus \{o\}} \widehat{g}_x(z) \overline{\widehat{g}_y(z)} 
\]
and $const$ is a dimensional constant which does not depend on $x,y$. 
By abuse of notation we will write 
\[
\langle \Xi^{\star}_n , f\rangle = \int_{\T^d} \left (\sum_{z\in \T^d_n} \1_{B(z,1/2n) }(x) \left [ \sum_{w\in \Z^d_n} g^{\star}(nz,w)\sigma(w) \right] \right ) f(x) dx .
\]

\noindent
The proofs consist of two main steps.\\

\noindent
\textbf{Step 1:}  Tightness of $(\mathcal{L}(a_n^{\star} \Xi_n^{\star}))_{n\in \N}$ in $H_{-\epsilon}(\T^d)$ for $-\epsilon < -d/2$. \\

\noindent
\textbf{Step 2:} From the first step we argue that there exists $(n_k)_{k\in \N}$ such that 
\[
a^{\star}_{n_k} \Xi^{\star}_{n_k} \overset{\mathcal{L}} \longrightarrow \Xi^{\star}
\]
as $k\rightarrow \infty$ in $H_{-\epsilon}(\T^d)$. It remains to show that the limit is unique. It suffices by Section 2.1 of \citet{ledoux:talagrand} to prove that for $f\in \mathcal{F}$ we have that the characteristic functional
\begin{equation}\label{lim_chara}
\lim_{n\rightarrow \infty} \mathbb{E} (\exp(i \langle \Xi^{\star}_n, f \rangle)) = \Phi(f)
\end{equation}
 is the correct one with $\Phi$ is defined in \eqref{chara}. Let us give more details in the sequel.\\

\noindent
\textbf{Tightness:}
Let us give a sketch of the arguments used in order to prove tightness. Compare Section 4.2. in \citet{CHR17}, Section 3.1. from \citet{JanFourier} and Section 4 in \citet{LongRange}. 

We will use Rellich's theorem (\citet{Roe}) for showing tightness.
\begin{theorem}[Rellich's theorem]\label{thm-rellich}
    If $k_1<k_2$ the inclusion operator 
    $H^{k_2}(\mathbb{T}^d)\hookrightarrow H^{k_1}(\mathbb{T}^d)$ 
    is a compact linear operator.  
    In particular for any radius $R>0$, the closed ball 
    $\overline{B_{ H_{-\frac{\varepsilon}{2}}}(0,\,R)}$ is compact in 
    $H_{-\varepsilon}$.
\end{theorem}

Pick $k_1=-\epsilon$ and $k_2=-\frac{\epsilon}{2}$ and choose $-\epsilon > -\frac{d}{2}$.
Observe that
\[
    \| a^{\star}_n{\Xi^{\star}_{n}}\|_{L^2(\mathbb{T}^d)}^2
    = (a^{\star}_n)^2 \sum_{x,\,y\in \mathbb{Z}^d_n}
    g^{(\star)}(x,y)\sigma(x)
    \sum_{x',\,y'\in \mathbb{Z}^d_n }g^{(\star)}(x',y')\sigma(x')
\]
is a.s. finite, as, for any fixed $n$, it is a finite combination of essentially bounded random variables. Therefore 
$\Xi^{\star}_{n}\in L^2(\mathbb{T}^d) \subset 
 H_{-\varepsilon}(\mathbb{T}^d)$  a.s. It will be enough 
to show that, for all $\delta>0$, there exists a $R=R(\delta)>0$ such that
\[
     \sup_{n\in \mathbb{N}}
     \mathbb{P} 
     \Big(\| a^{\star}_n\Xi^{\star}_{n}\|_{H_{-\frac{\varepsilon}{2}}}\geq R\Big)\le \delta.
 \]
Using Markov's inequality and exploiting the representation of the field $\Xi_n^{\star}$  in Fourier space, we can bound after some longer computations

\[
 \sup_{n\in \mathbb{N}} \mathbb{E} 
 \Big[\| a^{\star}_n \Xi^{\star}_{n}\|_{ H_{-\frac{\varepsilon}{2}}}^2\Big]
 \le C.
\]

\noindent
\textbf{Finite-dimensional convergence.}
The first question is how do we guess the right scaling? To determine $a_n$ we do the following rough computation for (A-ind) and general $\star$: pick $x,y\in \T^d_n$ and compute the covariance of the odometer

\[
\begin{split}
\mathbb{E}(u^{\star}_{\infty}(n x) u^{\star}_{\infty}(n y)) &= \sum_{w \in \Z^d_n} g^{\star}(xn,w)  g^{\star}(w,yn) \\
& \approx n^d \sum_{w\in \Z^d_n \setminus o} \widehat{g^{\star}}_{xn}(w) \overline{ \widehat{g^{\star}}_{yn}(w)} \\
& \approx n^{-d} \sum_{w\in \Z^d_n \setminus o} \frac{e^{2\pi in w(x-y)}}{|\lambda^{\star}_w|^2}
\end{split}
\]

Now using roughly that 

\begin{equation}\label{ev}
\lambda^{\star}_w \approx \begin{cases} \left \| \frac{w}{n}\right \|^2 & \text{ if }  \star=n.n \\ 
\left \| \frac{w}{n}\right \|^{2}\log(n) & \text{ if }  \star=l.r, \gamma= 2 \\
\left \| \frac{w}{n}\right \|^{\gamma} & \text{ if }  \star=l.r, \gamma=\min\{2,\alpha\}, \gamma\neq 2\\
\end{cases}
\end{equation}

we argue 
\[
\mathbb{E}(u^{\star}_{\infty}(n x) u^{\star}_{\infty}(n y)) \approx  \begin{cases} 
 n^{4-d} \sum_{w\in \Z^d_n \setminus o}\frac{e^{2\pi in w(x-y)}}{\|w\|^4} & \text{ if }  \star=n.n \\ 
n^{2-d}\log(n)  \sum_{w\in \Z^d_n \setminus o} \frac{e^{2\pi in w(x-y)}}{\|w\|^{2}} & \text{ if }  \star=l.r, \gamma=2 \\
n^{\gamma-d}  \sum_{w\in \Z^d_n \setminus o} \frac{e^{2\pi in w(x-y)}}{\|w\|^{\gamma}} & \text{ if }  \star=l.r, \gamma=\min\{2,\alpha\}, \gamma\neq 2.
\end{cases}
\]
For proving convergence to the characteristic functional we will first prove moment convergence for initial $\sigma$'s whose moments exist. We state the following Proposition 4.16 from \citet{LongRange}. 

\begin{proposition}\label{prop-convergence-moments-bounded-case}
    Assume $\mathbb{E}[\sigma]=0, \mathbb{E}[\sigma^2]=1$ and that 
    $ \mathbb{E}|\sigma|^k < \infty$ for all $k \in \mathbb{N}$. Then for all $m \geq 1$
    and for all $f \in C^\infty (\mathbb{T}^d)/\sim$ with zero average, the following 
    limits hold:
    \begin{equation}
        \lim_{n \to \infty} \mathbb{E}[\langle a^{\star}_n\Xi^{\star}_{n}, f \rangle^m ]
        =
        \begin{cases}
            (2m-1)!! [\mathbb{E}(\langle \Xi^{\star}, f \rangle^2)]^{m/2}, & m \in 2\mathbb{N} \\
            0, & m \in 2\mathbb{N} +1.
        \end{cases}
    \end{equation}
\end{proposition}

Note that it is more general than the corresponding Proposition 13 in \citet{CHR17} which assumes bounded $\sigma$'s almost surely instead of finite moments. 
Let us focus on the second moment and give a sketch for the proof for higher moments.
For the second moment observe that for $f\in \mathcal{F}$,
\begin{equation}\label{sec_mom}
\begin{split}
\mathbb{E}\left[\la\Xi^{\star}_n,\,f \ra^2 \right]&=  n^{-2d}(a^{\star}_n)^2 \sum_{z,\,z'\in \T_n^d} f(z)f(z') \mathbb{E}[ \eta^{\star}(nz) \eta^{\star}(nz')]+\mathbb{E}\left[R_n(f)^2\right]\\
&+\mathbb{E}\left[n^{-d}a_n^{\star} \sum_{z\in \T_n^d} f(z)\eta^{\star}(nz)R_n(f)\right ]
\end{split}
\end{equation}
where the reminder is equal to
\[
R_n(f):= a_n^{\star} n^{-d} \sum_{z\in \T_n^d} \eta^{\star}(nz)\left(\int_{B(z,\,\frac{1}{2n})}n^d f(x) dx-f(z)\right).
\]
The key observation here is that the first term on the R.H.S. of \eqref{sec_mom} gives the desired variance and the remaining terms converge to 0 in $L^2(\T^d)$. Indeed we have,
\begin{proposition}
For any $f\in \mathcal{F}$, 
\[
\lim_{n\rightarrow \infty}\mathbb{E}\left[R^2_n(f)\right ] =0, \, \, \lim_{n\rightarrow \infty}\mathbb{E}\left[n^{-d}a_n^{\star} \sum_{z\in \T_n^d} f(z)\eta^{\star}(nz)R_n(f)\right ]=0.
\]
\end{proposition}
The previous proposition corresponds to Propositions 6 in \citet{CHR17}, 4.16 in \citet{LongRange} and 9 in \citet{JanFourier}. The proof uses Cauchy-Schwartz inequality.
It remains to show that (Proposition 5 in \citet{CHR17}, 4.17 in \citet{LongRange} and Proposition 8 in \citet{JanFourier}).
\begin{proposition}\label{prop:var}
For $f\in \mathcal{F}$,
\[
\lim_{n\rightarrow \infty} n^{-2d} (a^{\star}_n)^2 \sum_{z,\,z'\in \T_n^d} f(z)f(z') \mathbb{E}[ \eta^{\star}(nz)\eta^{\star}(nz')] = \begin{cases} \| f \|^2_{-1}& \text{ if } \star=n.n., (A-ind) \\ 
\| f \|^2_{K, -1}& \text{ if } \star=n.n., (A-cor) \\
\| f \|^2_{-\gamma/2}& \text{ if } \star=l.r., (A-ind), \gamma=\{2,\alpha\}.
\end{cases}
\]
\end{proposition}

\begin{proof}
To prove  Proposition \ref{prop:var} we write:
\begin{equation}\label{var_comp}
\begin{split}
n^{-2d} (a^{\star}_n)^2 \sum_{z,\,z'\in \T_n^d} f(z)f(z') \mathbb{E}[ \eta^{\star}(nz)\eta^{\star}(nz')]  & = n^{-2d} (a^{\star}_n)^2 \sum_{z,\,z'\in \T_n^d} f(z)f(z') \sum_{w\in \Z^d_n} \frac{e^{2\pi i(z-z')\cdot w}}{|\lambda_w^{\star}|^2}.
\end{split}
\end{equation}
Observe that
\[
n^{-2d}\sum_{z,z'\in \T^d_n} f(z) f(z') \exp \left ( 2\pi i (z-z')w\right ) = | \widehat{f}(w) |^2. 
\]
Under the assumption (A-ind) we can use the asymptotics of the eigenvalues \eqref{ev} which will be of order $a^{\star}_n$ to make the convergence work. 

We need to adjust our proofs to the case where
\[
\sum_{w\in \Z^d\setminus \{o\}} \frac{1}{\|w\|^{4}} \, \, \, \text{ resp. } \sum_{w\in \Z^d\setminus \{o\}} \frac{1}{\|w\|^{2\gamma}}
\]
are converging or not. For $d<d_c^{\star}$ which was $d_c^{n.n.}=4$ when $\star=n.n.$ and $d_c^{l.r.}=2\gamma$ in the long-range case the sum converges and the argument is straightforward. In the other case we need a mollifying procedure. 
Let $\rho \in \mathcal{S}(\R^d)$ be a normalized function in the Schwartz space with support in $\T^d$. Let $\rho_k(x) =\frac{1}{k^p}\rho(x/k)$ then for all $m\in \N$ there exists $C>0$ such that 
\[
|\widehat{\rho}_k(w)| \leq \frac{c}{(1+\|w\|)^m}.
\]
It remains to prove that the mollified version converges to the desired variance
\begin{equation}\label{eq-mollified-goes-to-norm}
    \lim_{\kappa \to 0^+} \lim_{ n \to \infty} 
    n^{-2d}
    \sum_{z,z^\prime \in \mathbb{T}^d_n}f(z)f(z^\prime) 
    \sum_{w \in \mathbb{Z}^d_n \backslash \{0\}} \widehat{\rho_\kappa}(w)
    \frac{\exp(2\pi i (z-z^\prime)\cdot w )}{\|w\|^{\star}}= 
    \mathbb{E}(\langle \Xi^{\star},f \rangle^2),
\end{equation}
and the mollifier correction disappears,  
\begin{equation}\label{eq-error-in-mollification}
    \lim_{\kappa \to 0^+} \overline{\lim_{ n \to \infty}}
    \Big|
    n^{-2d}
\!\!\!
    \sum_{z,z^\prime \in \mathbb{T}^d_n}
\!\!\!\!    
    f(z)f(z^\prime) 
\!\!\!\!\!
    \sum_{w \in \mathbb{Z}^d_n \backslash \{0\}} 
\!\!\!\!\!
    \Big(1-\widehat{\rho_\kappa}(w)\Big)
    \frac{\exp(2\pi i (z-z^\prime)\cdot w )}{\|w\|^{\star}}
    \Big|
    =0.
\end{equation}

\end{proof}

Let us remark that in the correlated case (A-cor) we need only to worry about the second moment since the initial configuration is Gaussian. The variance will depend additionally on the Fourier multiplier $\widehat{K}$. 

The most involved part in the proofs is to control that the error introduced in \eqref{var_comp}, by replacing the eigenvalues $\lambda_w^{\star}$ by the leading order \eqref{ev}, is vanishing. In the long-range case this poses an extra difficulty since it requires a fine analysis to get to the right leading order and additionally to make the convergence work. 

Higher moments follow from the following combinatorial argument. Take a fixed $f \in \mathcal{F}$ and define a map $T_n: \mathbb{T}^d \longrightarrow \mathbb{R}$ by 
\begin{equation}\label{map_T}
z \longmapsto \int_{B(z, \frac{1}{2n})}    f(y) dy. 
\end{equation}
For $m\geq 3$ denote by $\mathcal{P}(m)$ the set of all partitions of $\{1,2,...,m\}$ and $\Pi$ the elements of a partition $P\in \mathcal{P}(m)$. Then we write
\[
\mathbb{E}(\langle \Xi_n^{\star},f\rangle^m) =  \sum_{P \in \mathcal{P}(m) } \prod_{\Pi \in P}
    (a_n^{\star})^{|\Pi|}
    \mathbb{E}[\sigma^{|\Pi|}] 
    \sum_{x \in \mathbb{Z}^d_n}\Big( 
        \sum_{z \in \mathbb{T}^d_n} 
        g^{(\alpha)}(x,nz)T_n(z)
    \Big)^{|\Pi|}    
\]
and use again properties of the eigenvalues, independence of the $\sigma$'s to prove that it converges to the right moment of the corresponding Gaussian model, see Proposition \ref{prop-convergence-moments-bounded-case}. We stress that independence of the $\sigma$ variables is a key property here. In fact, if we want to consider correlated random variables, then the higher moment proof is much more involved.We expect that if the correlations are small then the proof should still work.\\

\noindent
\textbf{Truncation method:} The proof will be completed once we show the truncation method.
Fix an arbitrarily large (but finite) constant $R>0$ and set
\begin{align}
w_n^{<R}(x)&:=\frac{1}{2d}\sum_{y\in \Z_n^d}g^{\star}(x,\,y)\sigma(y)\1_{\{|\sigma(y)|< R\}},\\
w_n^{\geq R}(x)&:=\frac{1}{2d}\sum_{y\in \Z_n^d}g^{\star}(x,\,y)\sigma(y)\1_{\{|\sigma(y)|\geq R\}}.
\end{align}
Clearly $w_n(\cdot)=w_n^{<R}(\cdot)+w_n^{\geq R}(\cdot)$. To prove our result, we will use
\begin{theorem}[Theorem~4.2 \citet{Bi68}] \label{thm:billy}
Let $S$ be a metric space with metric $\rho$. Suppose that $(X_{n,\,u},\,X_n)$ are elements of $S\times S.$ If
\[
\lim_{u\to+\infty} \limsup_{n\to+\infty} \mathbb{P}\left(\rho(X_{n,\,u},\,X_n)\geq \tau\right)=0
\]
for all $\tau>0$, and $X_{n,\,u}\Rightarrow_{n}Z_u\Rightarrow_u X$, where $``\Rightarrow_x''$ indicates convergence in law as $x\to +\infty$, then $X_n\Rightarrow_n X$.
\end{theorem}
Following this Theorem, we need to show two steps:
\begin{itemize}
\item[(Step 1)] $\lim_{R\to+\infty} \limsup_{n\to+\infty} \mathbb{P}\left(\left\| a^{\star}_n\Xi^{\star}_{n}- a^{\star}_n\Xi^{\star}_{w_n^{<R}}\right\|_{ H_{-\epsilon}}\geq \tau\right)=0$ for all $\tau>0$.
\item[(Step 2)] For a constant $v_R>0$, we have $a^{\star}_n \Xi^{\star}_{w_n^{<R}}\Rightarrow_n \sqrt{v_R}\, \Xi^{\star} \Rightarrow_{R}\Xi$ in the topology of $H_{-\epsilon}$.
\end{itemize}
As a consequence we will obtain that $a^{\star}_n\Xi^{\star}_{n}$ converges to $\Xi^{\star}$ in law in the topology of $H_{-\epsilon}.$
To prove (Step 1) notice that
\[
\left\| a^{\star}_n\Xi_{n}- a^{\star}_n\Xi_{w_n^{<R}}\right\|_{H_{-\epsilon}}=\left\| a^{\star}_n \Xi_{w_n^{\geq R}}\right\|_{H_{-\epsilon}}
\]
by definition, for every realization of $(\sigma(x))_{x\in \Z_n^d}$. Since, for every $\tau>0$,
\[
\mathbb{P}\left(\left\| a^{\star}_n \Xi_{w_n^{\geq R}}\right\|_{ H_{-\epsilon}}\geq\tau\right)\le\frac{\mathbb{E}\left[\left\| a^{\star}_n \Xi_{w_n^{\geq R}}\right\|_{ H_{-\epsilon}}^2\right]}{\tau^2}
\]
it will suffice to show that the numerator on the right-hand side goes to zero to show (Step 1).  The second step follows from the next considerations.
Set $v_R = Var(\sigma\1_{\{|\sigma|<R\}})$ and
\[
\sigma^R(x):=\sigma(x)\1_{\{|\sigma(x)|<R\}}-m_R.
\]
Let us look at the following fields, $c>0$ some constant, 
\[
\Xi^{\star}_{n,\,R}(x):= c \cdot a^{\star}_n \sum_{z\in \T^d_n}\sum_{w\in\Z_n^d}g^{\star}(w,\,nz)\sigma^R(w)\1_{B(z,\,\frac{1}{2n})}(x),\quad x\in \T^d.
\]
Since $\sum_{y\in \Z_n^d}g^{\star}(\cdot,\,y)$ is a constant function on $\Z_n^d$ it follows that
\[
\la \Xi^{\star}_{n,\,R},\,f  \ra=\la \Xi^{\star}_{w_n^{<R}},\,f\ra
\]
for all smooth functions $f$ with zero average. Therefore the field $\Xi^{\star}_{n,\,R}$ has the same law of $\Xi^{\star}_{w_n^{<R}}$. If we multiply and divide the former expression by $\sqrt{v_R}$ and take the limits we get the claim. 


\section{Alpha-Stable Distributions}\label{sec:stable}

Which assumptions on the initial distribution or redistribution of the mass are responsible for the Gaussian nature of the limiting odometer interface?
We have seen that the redistribution (nearest neighbours versus long-range)  of the mass will influence the regularity of the limiting field but not the Gaussianity. In fact, the finite second moment assumption of the $\sigma$'s is responsible for the CLT-type behaviour. 

In \citet{CHRheavy} the authors proved that for $\sigma$'s which are in the domain of attraction of stable distributions the odometer scales to an $\alpha$-stable continuum field on $\T^d$. To the authors knowledge it is the first time that a continuum $\alpha$-stable field was constructed. 

\subsection{Stabilization on infinite graphs}

In \citet{CHRheavy} the authors investigate the dichotomy between stabilization and explosion for infinite vertex transitive graphs $G$ when $(s(x))_{x\in V}$ is i.i.d. and has not necessarily a finite mean or variance.

In Lemma 1 it is shown that for $\mathbb{E}(s(o))=\infty$ the sandpile will not stabilize almost surely while for $\mathbb{E}(s(o))=-\infty$ it does.
Assuming finite mean of the initial configuration and infinite variance one has to be more careful. As in \citet{LMPU} one has to distinguish two cases, when $\sum_{y\in V}g(0,y)^{\alpha}=\infty$ and when it is finite. The first case is called the \textit{doubly transient case} and corresponds to e.g. $\Z^3,\Z^4$ and the latter the \textit{singly transient case} corresponding for example to $\Z^d$ for $d\geq 5$. 

In both cases it turns out that for $\mathbb{E}(s(o))=1$ and $\sigma$'s regularly varying, the probability of stabilization is 0 (compare Theorems 2 and 3 from \citet{CHRheavy}). The same result holds true in the finite variance case, proven already by \citet{LMPU} in Theorem 1.1.

\subsection{Scaling limit of the odometer for heavy-tail inital distributions}

Let us first state assumptions on the initial configuration $s$. We will again assume that $s$ satisfies
\[
s(x) = 1 + \sigma(x) - \frac{1}{n^d}\sum_{z\in \Z^d_n}\sigma(z).
\]

 A non-negative random variable $X$ is called {\em regularly varying} of index $\alpha\ge 0$, and we write $X\in RV_{-\alpha}$, if
\[
 \mathbb{P}(X>x) \sim x^{-\alpha} L(x) \quad \text{as }x \rightarrow +\infty 
\]
where $L$ is a slowly varying function, i. e.,
\[
\lim_{x\rightarrow +\infty}\frac{L(tx)}{L(x)}=1\quad\text{for all }t>0.
\]
We recall the definition of random variables which fall in the domain of attraction of $\alpha$-stable distributions:
\begin{definition}[Domain of normal attraction of stable variables]
\label{firstassumption}
Let $\alpha\in (0,\,2]$. Let $V$ be a countably infinite index set and $(W(x))_{x\in V}$ be i.i.d. symmetric random variables with common distribution function in the domain of normal attraction of an $\alpha$-stable distribution. This means that, for $V_1\subset V_2\subset\ldots$ such that $\cup_{k\geq 1}V_k=V$, we have the following limit:
\begin{equation}\label{eq:doa:stable}
\lim_{k\rightarrow +\infty}{|V_k|^{-\frac{1}{\alpha}}}\sum_{x\in V_k} W(x)\stackrel{d}{=} \rho_\alpha,
\end{equation}
where $\rho_\alpha$ has a symmetric $\alpha$-stable law which we denote as $S\alpha S(\mathfrak c)$, that is, $\mathbb{E}[ \exp(i \theta \rho_\alpha)]= \exp({-\mathfrak{c}^\alpha|\theta|^\alpha})$ for some $\mathfrak{c}\in \mathbb{R}$. 
\end{definition}
The scaling limit result (Theorem 5) in \citet{CHRheavy} is proven in the setting  $V_n:=\Z_n^d$ and $V= \Z^d$.  If the scale parameter of the $\alpha$-stable law is $1$, we will write $\sigma(x)\stackrel{d}{=} S\alpha S(1)$. In this case, it is well known that $|\sigma(x)|$ has a regularly varying tail with index $-\alpha$, for $\alpha\in (0,2]$.
The results we are going to prove can be extended to a more general set-up assuming further necessary and sufficient conditions for the $(\sigma(x))_{x\in V}$ to be in the domain of attraction of stable variables, see also \citet{ST}. 
Proposition \ref{prop:dislaw} is valid also for finite connected graphs $G=(V,E)$, compare Proposition 4 in \citet{CHRheavy}.\\

\noindent
\textbf{Example:}
For $\sigma$'s chosen from a Cauchy distribution or Pareto with parameter 1.5 (Figure \ref{Cauchy} resp. \ref{Pareto}) we see that the odometer interface has a completely different shape. Indeed, the limiting interface will not be Gaussian any more.

\begin{figure}[h]
\centering
 \centering
    \begin{minipage}{.5\textwidth}
\includegraphics[scale=0.5]{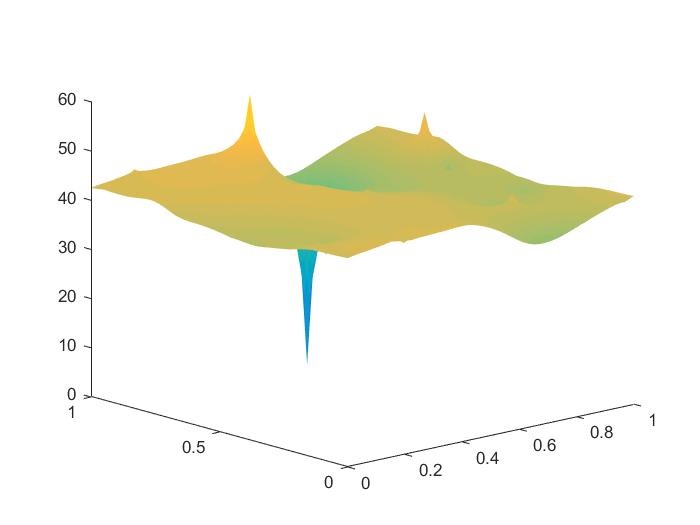}
\caption{Realization of the odometer interfaces for Cauchy initial $\sigma$'s}
\label{Cauchy}
\end{minipage}%
    \begin{minipage}{.5\textwidth}
\centering
\includegraphics[scale=0.5]{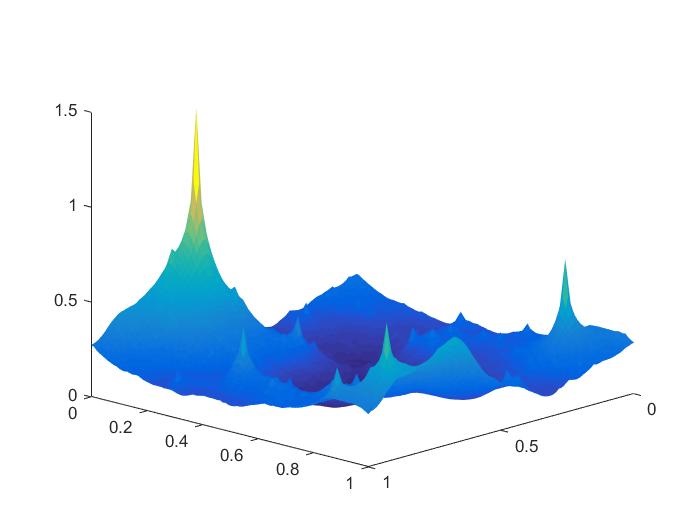}
\caption{Realization of the odometer interfaces for Pareto $\sigma$'s with parameter 1.5}
\label{Pareto}
\end{minipage}
\end{figure}

In the following we define the continuum $\alpha$-stable field (compare Definition 2.1. in \citet{KumMan}).
\begin{definition}
A random distribution (field) $\Xi_{\alpha}$ on the dual space of smooth functions on $\T^d$, $(C^{\infty}(\T^d)/\sim)^*$, is called $\alpha$-stable if, given two independent copies $\Xi_{\alpha,\,1}$ and $\Xi_{\alpha,\,2}$ of $\Xi_\alpha$, then for any $a,\, b>0$ and $f \in (C^{\infty}(\T^d)/\sim)^*$
\[
\mathbb{E}[\exp(i \la \Xi_{\alpha,\,1}, af\ra )]\mathbb{E}[\exp(i \la \Xi_{\alpha,\,2}, bf\ra )]= \mathbb{E}\left[\exp\left(i \la \Xi_{\alpha}, (a^\alpha+b^\alpha )^{\frac{1}{\alpha}} f \ra \right)
\right].
\]
\end{definition}
Alternatively we can represent the above characteristic functional of the field as
\[
\mathbb{E}(\exp(i \langle \Xi_{\alpha},f\rangle)) = \exp \left ( -\|\Delta^{-1} f\|^{\alpha}_{L^{\alpha}(\T^d)} \right ).
\]

The following Theorem 5 from \citet{CHRheavy} states the finite-dimensional convergence of the odometer field. 

\begin{theorem}\label{thm:scaling_RV}
Let $d\ge 1$,  $(\sigma(x))_{x\in \Z_n^d}$ be i.i.d. and satisfy Definition \ref{firstassumption}, and furthermore let $(s(x))_{x\in \Z_n^d}$ satisfy \eqref{ini_s}. Consider $a_n=4\pi^2 n^{d-d/\alpha -2}$.
There exists a random distribution $\Xi_\alpha$ on $(C^\infty(\T^d))^\ast$ such that: for all $m\in \N$ and $f_1, f_2,\ldots,\, f_m\in C^\infty(\T^d)$ with mean zero, the random variables $\la a_n\Xi_n, f_j \ra$ converge jointly in distribution to a random variable $\la a_n \Xi_\alpha, f_j\ra$. Moreover, the characteristic functional of $\Xi_\alpha$ is given by 
\begin{equation}\label{eq:limitfield}
\mathbb{E}[\exp(i \la \Xi_{\alpha}, f\ra )]=\exp\left(-\int_{\T^d} \left| \sum_{z\in \Z^d\setminus \{0\}} \frac{\exp(-2\pi i z\cdot x)}{\|z\|^2} \widehat{f}(z)\right|^{\alpha} dx \right).
\end{equation}
\end{theorem}

\begin{proof}
Let us present a sketch of the proof and the main key steps. We will again rely on Fourier analysis since we need to determine a characteristic functional. Note that we are not working any more on AWS and Sobolev spaces but rather dual spaces of smooth functions on the torus. Therefore we do not have Rellich's theorem at hand and cannot prove tightness as in the previsous cases. 

 However, since we are dealing now with random fields with possibly no moments, we cannot use any more a moment convergence argument. What will be helpful is to consider a simpler case. We will first consider $\sigma\sim S\alpha S(1)$ and then look at the general case. 

The proof is decomposed into 5 steps and they are related to each other by: 
\begin{center}
Step~\ref{step:5}$\Rightarrow$ Step~\ref{step:4}$\Rightarrow$ Step~\ref{step:3}$\Rightarrow$
Step~\ref{step:2}$\Rightarrow$
Step~\ref{step:1}.
\end{center}

Similar as before, due to mean zero test functions, we can write
\begin{equation}\label{Xi_alpha}
\langle \Xi_{n}, f \rangle = \sum_{x\in \Z^d_n} \left ( \sum_{z\in \T^d_n} g(x,nz)T_n(z)\right )\sigma(x)=:\sum_{x\in \Z^d_n} k_n(x)\sigma(x)
\end{equation}
where $T_n$ was defined in \eqref{map_T}. Then we get
\[
\mathbb{E}(\exp(i\langle a_n \Xi_n,f \rangle)) =\exp \left(-\sum_{x\in \Z^d_n} |a_n k_n(x)|^{\alpha} \right).
\]

Defining a reminder term
\[
R_n(w):= \int_{B\left(w,\,\frac{1}{2n}\right)} (f(u)-f(w)) d u
\] 
which will be small in the limit we can rewrite $k_n$:

\begin{align}
k_n(x)&=-\frac{1}{n^d}\sum_{w\in\T_n^d}\left( n^{-d} f(w)+R_n(w)\right) \sum_{z\in \Z_n^d\setminus\{0\}} \frac{ e^{2\pi i z(x-nw)} }{\lambda_z}\nonumber\\
&=-\frac{1}{n^{2d}}\sum_{w\in\T_n^d} f(w)\!\sum_{z\in \Z_n^d\setminus\{0\}} \frac{ e^{2\pi i z(x-nw)}}{\lambda_z}-\frac{1}{n^d}\sum_{w\in\T_n^d}R_n(w)\!\sum_{z\in \Z_n^d\setminus\{0\}} \frac{ e^{2\pi i z(x-nw) }}{\lambda_z}\nonumber\\
&:= l_n(x)+ C_n(x)\label{eq:def_lC}.
\end{align}
First one shows that the convergence of $\exp\left(-\sum_{x\in \Z_n^d} |a_n k_n(x)|^\alpha\right)$ can be given in terms of the same quantity where $k_n(\cdot)$ is replaced by $l_n(\cdot)$:
\begin{step}\label{step:1}
\[
\lim_{n\rightarrow +\infty}\left|\exp\left(-\sum_{x\in \Z_n^d} |a_n k_n(x)|^\alpha\right)- \exp\left(-\sum_{x\in \Z_n^d} |a_n l_n(x)|^\alpha\right)\right|=0.
\]
\end{step}
In the next steps one proves that $l_n$ yields the correct characteristic function. In Step~\ref{step:2} we are introducing a mollifier which will help to extend sums from $\Z_n^d$ to the whole lattice. 
\begin{step}\label{step:2}
{Let $\phi\in \mathcal S(\R^d)$, the Schwartz space, with
$\int_{\R^d}\phi(x) d x=1$. Let $\epsilon>0$ and let $\phi_\epsilon(x):= \epsilon^{-d}\phi\left({x}{\epsilon}^{-1}\right)$ for $\epsilon>0$.} Then 
\begin{align*}
\lim_{\epsilon\downarrow 0}\lim_{{n\rightarrow +\infty}}\left|\sum_{x\in \Z_n^d} |a_n l_n(x)|^{\alpha}- \frac{a_n^\alpha}{n^{d\alpha}} \sum_{x\in \T_n^d}\left| \sum_{z\in \Z_n^d\setminus \{0\}} \frac{\widehat \phi_\epsilon(z)\exp(-2\pi \imath z\cdot x) }{\lambda_z} \widehat{f_n}(z)\right|^{\alpha}\right|=0
\end{align*}
where $\widehat{f_n}(z)= n^{-d}\sum_{w\in \T_n^d} f(w)\exp(2\pi i w\cdot z)$. 
\end{step}
The goal of the third step is to approximate each eigenvalue $\lambda_z$ of the Laplacian with the norm of the point $z$, namely
\begin{step}\label{step:3}
Uniformly in $\epsilon>0$
\begin{align*}
\lim_{n\to+\infty}\frac{a_n^\alpha}{n^{d\alpha}} &\left|\sum_{x\in \T_n^d}\left[\left| \sum_{z\in \Z_n^d\setminus \{0\}} \frac{\widehat \phi_\epsilon(z)\e^{-2\pi \imath z\cdot x} }{\lambda_z} \widehat{f_n}(z)\right|^{\alpha}
\right.\right.\\
&\left.\left.-\frac{ n^{2\alpha}}{4^\alpha {\pi^{2\alpha}} } \left| \sum_{z\in \Z_n^d\setminus \{0\}} \frac{\widehat \phi_\epsilon(z)\e^{-2\pi \imath z\cdot x} }{\|z\|^2} \widehat{f_n}(z)\right|^{\alpha}\right]\right|=0
\end{align*}
\end{step}
In the next step we extend the sums in Step \ref{step:3} over $\Z^d$ using the decay of the mollifier.
\begin{step}\label{step:4}
Uniformly in $\epsilon>0$
\begin{align*}
\lim_{n\rightarrow +\infty}\frac{a_n^\alpha n^{2\alpha}}{n^{d\alpha}{4^\alpha\pi^{2\alpha}}}&\left| \sum_{x\in \T_n^d}\left| \sum_{z\in \Z_n^d\setminus \{0\}} \frac{\widehat \phi_\epsilon(z)\exp(-2\pi \imath z\cdot x) }{\|z\|^2} \widehat{f_n}(z)\right|^{\alpha}\right.\\
& -
\left. \sum_{x\in \T_n^d}\left| \sum_{z\in \Z^d\setminus \{0\}} \frac{\widehat \phi_\epsilon(z)\exp(-2\pi \imath z\cdot x) }{\|z\|^2} \widehat{f_n}(z)\right|^{\alpha}\right|=0.
\end{align*}
\end{step}
At last, we can finally show the convergence of the sum to the required integral.
\begin{step}\label{step:5}
\eq{}\label{eq:ice}\lim_{\epsilon\downarrow 0}\lim_{{n\to+\infty}}\frac{1}{n^d}\sum_{x\in \T_n^d}\left| \sum_{z\in \Z^d\setminus \{0\}} \frac{\widehat \phi_\epsilon(z)\e^{-2\pi \imath z\cdot x} }{\|z\|^2} \widehat{f_n}(z)\right|^{\alpha}
= \int_{\T^d} \left| \sum_{z\in \Z^d\setminus \{0\}} \frac{\e^{-2\pi \imath z\cdot x}}{\|z\|^2} \widehat{f}(z)\right|^{\alpha} \De x.\eeq{}
\end{step}

Those steps are quite involved to prove since we cannot rely on moment assumptions and need to perform fine analysis. 
It remains to prove  Theorem \ref{thm:scaling_RV} for regularly varying initial distributions. This will be implied by Proposition 12 of \citet{CHRheavy} stating that for all $f\in \mathcal{F}$, $\epsilon>0$

\[
\lim_{n\rightarrow \infty} \mathbb{P} \left ( | \langle \Xi_n,f \rangle - \langle \Xi'_n,f \rangle| > \epsilon \right )=0
\]
where
\[
\langle \Xi'_{n}, f \rangle = \sum_{x\in \Z^d_n} \left ( \sum_{z\in \T^d_n} g(x,nz)T_n(z)\right )\sigma(x)  
\]
and $\sigma \sim \rho_{\alpha}$ from Definition \ref{firstassumption}.

\end{proof}

\section{Internal Diffusion Limited Aggregation}\label{sec:iDLA}

In this section we want to discuss the connection between odometer functions of divisible sandpile models and internal diffusion limited aggregation models (iDLA). 
 
It was introduced for the first time by \citet{DiaFul}. The iDLA model describes the growth of a random aggregate of particles from the inside out. Start initially with $n$ particles at the origin $o\in \Z^d$ and let each particle perform a random walk until it reaches an unoccupied site. Consider the random set of occupied sites $D(n)$. Bramson et al. in \citet{Brams} were the first ones to identify that $D(n^d)$ is very close to the Euclidean ball of radius $n$, which was later sharpened by \citet{Gaudi1,Gaudi2}. 

Levine and Peres in \citet{LevPerDLA} identified that the scaling limit of three growth models, namely the iDLA, rotor-router and divisible sandpile models on $\Z^d$ to be the same and are related to solutions to certain PDE free boundary problems in $\R^d$. 
To identify the limiting shape, the authors in \citet{LevPerDLA} consider the odometer function of the divisible sandpile. Recall that it satisfies the following relation (see Proposition \ref{prop:stab}):

\[
\Delta u(x) = s_{\infty}(x) - s(x)
\]

where $s_{\infty}$ is the final divisible sandpile configuration and $s$ the initial one. Due to the constraint, $\sum_{x\in \Z_n^d} s(z) =n^d$, we obtain that $s_{\infty}\equiv 1$. In general for $\sum_{x\in \Z^d_n} s(x)< n^d$ the limiting configuration $s_{\infty}$ is a random configuration depending on $s$.

Remember that an alternative way to find the odometer function is to solve the obstacle problem, see Lemma \ref{odo_lev}, where one has to find a proper obstacle $\gamma$ satisfying the Laplacian equation $(-\Delta) \gamma = s-1$ and a solution to the obstacle problem $v$ satisfying $v(x) = \inf \{ f(x)|f\geq \gamma; (-\Delta)f \leq 0 \}$. Then the odometer is equal to
\begin{equation}\label{obs_u}
u=v-\gamma.
\end{equation}
In fact, this was the approach chosen in \citet{frometa:jara} to determine the scaling limit of the odometer of a truncated long-range divisible sandpile in $\Z^2$.

Let us point out again a crucial fact that the scaling limit for the divisible sandpile, obtained in the previous sections, corresponds to the scaling limit of the \textit{obstacle} in this context and not of \eqref{obs_u}. This is due to the fact that the mean 0 test functions cancel out the contribution of $v$ which is the solution of the obstacle problem.

For finding the limiting shape of the odometer, one has first to find an appropriate obstacle $\gamma$. A good candidate will turn out to be
\begin{equation}\label{gamma}
\gamma(x) = -|x|^2 - \sum_{y\in \Z^d} G(x,y)s(y)
\end{equation}
where $G(x,y)$ is the Green's function in $\Z^d$ for $d\geq 3$ and the recurrent potential kernel in $\Z^2$. Chosen $\gamma$ in this way
we get that $\Delta (u+\gamma) \leq 0$ hence $u+\gamma$ is a superharmonic function on $\Z^d$. For any superharmonic function $f\geq \gamma$ 
\[
\Delta (f-\gamma - u) \leq 0
\]
on the domain $\overline{D}=\{ x\in \Z^d| s_{\infty}(x)=1 \}$ and non-negative outside $D$ hence non-negative everywhere. Hence we get that $u$ is equal to \eqref{obs_u}.

The description in the continuum $\R^d$ is then immediate. Given an initial mass $s\in \R^d$ and obstacle

\begin{equation}\label{gamma_r}
\gamma(x) = -|x|^2 - \int_{\R^d} G(x,y)s(y)dy
\end{equation}
where
\[
G(x,y) = \begin{cases} -\log(\|x-y \|) & \text{ if } d =2\\
\|x-y\|^{2-d} & \text{ if } d \geq 3
\end{cases}
\]
the odometer $u=v-\gamma$ where $v(x)=\{ f(x)| f \text{ is continuous, superharmonic }, f\geq \gamma\}$. 

\begin{definition}
The non-coincidence set for the obstacle problem with obstacle $\gamma$ is the domain of occupied sites given by
\[
D=\{ x\in \R^d| v(x) > \gamma(x)\}.
\]
\end{definition}
 
For a lattice spacing $\delta_n \rightarrow 0$ as $n \rightarrow \infty$ and domains $A_n \subset \delta_n \Z^d, D\subset \R^d$ write $A_n\rightarrow D$ if for any $\epsilon  >0$, $D_{\epsilon} \cap \delta_n \Z^d \subset A_n \subset D^{\epsilon}$ for $n$ large enough and $D^{\epsilon}=\{x\in \R^d|B(x,\epsilon) \not\subset D^c\}$ resp. $D_{\epsilon}= \{ x\in D| B(x,\epsilon) \subset D\}$. Let us state Theorem 1.2 from \citet{LevPerDLA} which identifies the non-coincidence set for the obstacle problem with the occupied sites of the iDLA, rotor-router and divisible sandpile model in the scaling limit.

\begin{theorem}\label{theo_shape}
 Let $B\subset \R^d$, $d\geq 2$ be a bounded open set. Put $s:\R^d \rightarrow \N\cup \{0\}$ bounded, continuous almost everywhere satisfying $\{\sigma \geq 1\} = \overline{B}$. Consider the initial configuration $s_{(n)} : \delta_n \Z^d \rightarrow \N\cup \{0\}$ with density 
\[
s_{(n)}(x) = \frac{1}{\delta_n^d} \int_{[x-\frac{\delta_n}{2}, x+\frac{\delta_n}{2}]^d} s(y) dy.
\]
Denote by $D^{\star}_n$ the domain of occupied sites for $\star\in \{ \text{div. sand.}, \text{iDLA}, \text{ ro.-rou.}\}$ respectively.
Then
\[
D^{\star}_n\rightarrow D \cup B
\]
for all $\star\in \{ \text{div. sand.}, \text{iDLA}, \text{ ro.-rou.}\}$ as $n\rightarrow \infty$ and $\delta_n \log(n)\rightarrow 0$ for $\star=div.sand.$ The convergence holds in the sense described above.
\end{theorem}

As a consequence all three models have rotational invariant scaling limits hence the limiting shape for the iDLA and divisible sandpile when starting with a pile of $n$ particles at the origin is a ball. 

\section{Discussion and open problems}\label{sec:dis}

In this last section we would like to give some perspectives and open questions.\\
 
\noindent
\textbf{Scaling limit of the odometer in the super- and subcritical case:}
Recall that the odometer $u$ is a solution of
\[
\Delta^{\star} u(x) = s_{\infty}(x) - s(x).
\]
The assumption on the initial configuration $\sum_{x\in \Z^d_n} s(x)=n^d$ ensured that $s_{\infty}\equiv 1$ and therefore 
\[
u = (\Delta^{\star})^{-1} (1-s)
\] 
depends only on the initial distribution of $s$. In the subcritical case $\sum_{x\in \Z^d_n} s(x)<n^d$ what is the distribution of $s_{\infty}$? How are $s_{\infty}$ and $s$ correlated? What is the scaling limit of $u$? For the supercritical case we know that the sandpile will not stabilize, but can we still control it and find a proper rescaling to make the odometer converge to a particular field? \\

\noindent
\textbf{Scaling limit of maxima:}
Concerning the scaling limit of the maxima note that the asymptotics of the mean odometer or expected maximum in Theorem \ref{max_odo_nn} agree with the asymptotics of the expected maximum of a bi-Laplacian model in $\Z^d$ when defining the model on a box $\Lambda_n\subset \Z^d$ of size $n$.
It was proven in Theorem 1 of \citet{CCH2015} that the scaling limit of extrema of bi-Laplacian models are Gumbel distributed for $d\geq 5$ and in \citet{Biltu} that the maximum coincides with the one of the GFF for $d=2,3$, see Corollary 2.2. The case $d=4$ was very recently solved by \citet{mem4d} where the author showed in Theorem 1.1. that the scaling limit is a randomly shifted Gumbel distribution. We conjecture that for $d>2\alpha$ the scaling limit is a Fr\'echet type distribution and at the critical point $d=2\alpha$ it is a randomly shifted Fr\'echet.\\

\noindent
\textbf{Limiting shapes for long-range divisible sandpiles:} In \citet{frometa:jara} the authors consider a truncated long-range divisible sandpile models in $\Z^2$.  Start with an initial configuration on $\Z^2$ with finite support and consider a divisible sandpile model. Topple when the mass $s(x)\geq 1$ and redistribute not only to nearest neighbours but also to further away neighbours according to a long-range random walk with $\alpha \in (1,2)$ truncated to stay within a fixed range $M$. The authors obtain the scaling limit of the odometer function by solving a corresponding obstacle problem.  Can one prove a similar result as Theorem \ref{theo_shape} for the (truncated) long-range divisible sandpile? What is the limiting shape? We conjecture that there exist a range of parameters $\alpha \in (\alpha_1, \alpha_2)$ where the limiting shape is not a ball.\\

\noindent
\textbf{Divisible sandpiles on different graphs:}
What are scaling limits of divisible sandpile models on different graphs? We know from \citet{LMPU} that for any finite graph as long as $\sum_{x\in \Z^d_n} s(x)\leq n^d$ then the divisible sandpile configuration will stabilize. In the infinite graph case for an initial i.i.d. configuration $s$, the divisible sandpile will stabilize if $\mathbb{E}(s(o))<1$. What about $\Z^d$ or other lattices or supercritical percolation (see also \citet{IDLA_perco})? We conjecture that the results will still hold true but require a much finer analysis and 
control over the convergence of  Green's functions. 


\bibliographystyle{plainnat}
\bibliography{literaturDSM}

\end{document}